\documentclass[11pt,leqno,draft]{amsart}

\usepackage{array,float,graphicx}
\usepackage{amssymb, amsmath, amsthm, mathrsfs}

\newtheorem{theorem}{Theorem}[section] 

\newtheorem{proposition}[theorem]{Proposition} 
\newtheorem{lemma}[theorem]{Lemma}

\theoremstyle{definition}
\newtheorem{definition}[theorem]{Definition}
\newtheorem{example}[theorem]{Example}
\newtheorem{problem}[theorem]{Problem}
\newtheorem{conjecture}[theorem]{Conjecture}
\newtheorem{exercise}{Exercise}

\theoremstyle{remark}
\newtheorem{remark}[theorem]{Remark}

\numberwithin{equation}{section}

\newtheorem*{acknowledgements}{Acknowledgements}

\def \ol {\overline}
\def \Q {\ensuremath{\mathbb{Q}}}
\def \C {\ensuremath{\mathbb{C}}}

\def \N {\ensuremath{\mathbb{N}}}

\def \mcE {\ensuremath{\mathcal{E}}}
\def \mcM {\ensuremath{\mathcal{M}}}

\def \mcR {\ensuremath{\mathcal{R}}}
\def \mcS {\ensuremath{\mathcal{S}}}
\def \mcU {\ensuremath{\mathcal{U}}}

\def \O {\mathcal{O}}

\def \Qp {\mathbb{Q}_p}
\def \Z {\mathbb{Z}}

\def \Zp  {\mathbb{Z}_p}

\def \Fp  {\mathbb{F}_p}
\def \Fq  {\mathbb{F}_q}

\def \bff {{\bf f}}

\def \bfG {{\bf G}}

\def \bfl {{\bf \ell}}
\def \bfm {{\bf m}}

\def \bfn {{\bf n}}
\def \mcN {\ensuremath{\mathcal{N}}}

\def \bfr {{\bf r}}

\def \bfX {{\bf X}}
\def \bfx {{\bf x}}
\def \bfY {{\bf Y}}
\def \bfy {{\bf y}}
\def \bfY {{\bf Y}}
\def \bfbeta {\boldsymbol{\beta}}

\def \L {\Lambda}
\def \GL {\text{GL}}

\def \wt {\widetilde}

\def \mcI {\mathcal{I}}

\def \zirr {\zeta^{\text{irr}}}

\def \T {\mathcal{T}}

\def \zirr {\zeta^{\rm{irr}}}
\def \G {\Gamma}

\def \tl {\triangleleft}

\def \nl {\triangleleft}
\def \sltwo {\mathfrak{sl}_2}
\def \SLn {\text{SL}_n}
\def \trnzp {\text{Tr}(n,\Zp)}
\def \trthreezp {\text{Tr}(3,\Zp)}
\def \ul {\underline}

\def \tuL {\textup{L}}
\def \bsalpha {\ensuremath{\boldsymbol{\alpha}}}
\def \bsbeta {{\boldsymbol{\beta}}}
\def \bsgamma {{\boldsymbol{\gamma}}}
\def \bsdelta {{\boldsymbol{\delta}}}

\begin{document}

\title[Zeta functions of groups and
rings]{A newcomer's guide to zeta functions of groups and rings}

\date{\today} \author{Christopher Voll} \address{Email: C.Voll.98 at
cantab.net } \address{School of Mathematics, University of
Southampton, Highfield, SO17 1BJ, United Kingdom.}


\begin{abstract}
These notes grew out of lectures given at the LMS\--EPSRC Short Course
on \emph{Asymptotic Methods in Infinite Group Theory}, University of
Oxford, 9-14 September 2007, organised by Dan Segal.
\end{abstract}

\maketitle

\tableofcontents

\section{Introduction}
\subsection{Zeta functions of nilpotent groups}
A finitely generated group $G$ has only finitely many subgroups of
each finite index. The zeta function of such a group is the Dirichlet
generating function encoding these numbers. If, for $m\in \N$, there
are $a_m=a_m(G)$ subgroups of index $m$ in $G$, the zeta function of
$G$ is defined as
\begin{equation}
\zeta_{G}(s):=\sum_{m=1}^\infty a_mm^{-s}=\sum_{H\leq_f
G}|G:H|^{-s}. \label{zeta definition}
\end{equation}
Here $s$ is a complex variable. Zeta functions were introduced into
infinite group theory as tools to study groups of \emph{polynomial
subgroup growth}. In fact, it is well-known that a series like
\eqref{zeta definition} converges on the complex right-half plane
$\{s\in\C|\;\mathfrak{Re}(s)>\alpha\}$ if and only if the numbers
$$s_m:=s_m(G):=\sum_{i\leq m}a_i$$ grow at most polynomially of degree
$\alpha$, i.e.~$s_m=O(1+m^\alpha)$. We therefore call
\begin{equation}\label{alpha definition}
\alpha_G:=\inf\{\alpha|\;\exists c>0\;\forall m:\;\sum_{i\leq
m}a_i<c(1+m^\alpha)\}
\end{equation}
the abscissa of convergence of $\zeta_G(s)$.\footnote{We have chosen
  this definition of $\alpha_G$ to ensure that $\alpha_G=-\infty$ if
  $G$ is finite.} We call groups with this property groups with
\emph{polynomial subgroup growth (PSG)}. The subgroup growth of any
group $G$ is the same as the subgroup growth of $G/R(G)$, where
$R(G):=\bigcap_{N\triangleleft_fG}N$, the finite residual of $G$, is
the intersection of the group's normal subgroups of finite index. In
studying subgroup growth, we may thus assume without loss of
generality that the group $G$ is residually finite, i.e.~that its
finite residual is trivial. Finitely generated, residually finite
groups of polynomial subgroup growth have been characterised as the
virtually soluble groups of finite
rank~(\cite{LubotzkyMannSegal/93}). This class of groups includes the
class of finitely generated, torsion-free nilpotent (or
$\T$-)groups. It was this class of PSG-groups for which zeta functions
were first introduced as a means to study asymptotic and arithmetic
aspects of subgroup growth (cf.~\cite{GSS/88}).

Let $G$ be a $\T$-group. It is not difficult to see that, owing to the
nilpotency of $G$, the zeta function $\zeta_G(s)$ has an Euler factorisation
\begin{equation}\label{euler zeta}
\zeta_G(s)=\prod_{p \text{ prime}}\zeta_{G,p}(s)
\end{equation}
into local (or Euler) factors $\zeta_{G,p}(s):=\sum_{i=0}^\infty
a_{p^i}p^{-is}$, indexed by the primes~$p$, enumerating subgroups of
$p$-power index. This generalises the familiar Euler product
decomposition satisfied by the Riemann zeta function
\begin{equation}\label{euler riemann}
\zeta(s):=\sum_{m=1}^\infty m^{-s}=\prod_{p\text{ prime}}\zeta_p(s),
\end{equation}
where $\zeta_p(s):=\frac{1}{1-p^{-s}}$. While~\eqref{euler riemann}
reflects the Fundamental Theorem of Arithmetic that every positive
integer can be written as the product of prime powers in an
essentially unique way, the identity~\eqref{euler zeta} reflects the
fact that every \emph{finite} nilpotent group is the direct product of
its Sylow $p$-subgroups. In fact, $\zeta(s)$ is the zeta function of
the infinite cyclic group, making~\eqref{euler riemann} a special case
of~\eqref{euler zeta}. Indeed, it is well known that, for all
$n\in\N$, there is a unique subgroup of index $n$ in $\Z$, namely
$n\Z$. It is instructive to see how this generalises to abelian groups
of higher rank.

\begin{example}\label{example abelian}
For $n\in\N$, let $\Z^n$ be the free abelian group of rank $n$. Then
\begin{equation}\label{formula abelian}
 \zeta_{\Z^n}(s)=\zeta(s)\zeta(s-1)\cdots\zeta(s-(n-1)).
\end{equation}
The monograph~\cite{LubotzkySegal/03} contains no fewer than five
proofs of this beautiful formula. We will add another, new one, in
Section~\ref{subsection local funeqs}. We observe that this formula
allows us to give precise asymptotic information about the numbers
$s_m(\Z^n)$ of subgroups of index at most~$m$ in $\Z^n$. Indeed, one
can deduce from \eqref{formula abelian} that
$$s_m(\Z^n)\sim n^{-1}\zeta(n)\zeta(n-1)\dots\zeta(2)m^n\quad\text{ as
}m\rightarrow\infty.$$ For example, using the identity
$\zeta(2)=\pi^2/6$, we see that
$$s_m(\Z^2)\sim (\pi^2/12)\,m^2\quad\text{ as }m\rightarrow\infty.$$
\end{example}

\subsection{Zeta functions of rings} 
We will see in Section~\ref{section linearisation} below that the
study of zeta functions of nilpotent groups may -- at least to a
certain extent -- be reduced to the study of zeta functions of
suitable rings. By a \emph{ring} we mean an additive group of finite
rank, carrying a bi-additive product, not necessarily commutative or
associative. Given a ring $L$, its (subring) zeta function is defined
as the Dirichlet generating series
$$\zeta_L(s)=\sum_{m=1}^\infty b_mm^{-s}=\sum_{H\leq_fL}|L:H|^{-s},$$
where, for $m\in\N$, $b_m=b_m(L)$ denotes the number of subrings of
index~$m$ in $L$ and $s$ is again a complex variable. By properties of
the underlying additive group of $L$ alone (essentially the Chinese
Remainder Theorem), this zeta function also satisfies an Euler product
decomposition
\begin{equation}\label{euler rings}
\zeta_{L}(s)=\prod_{p\text{ prime}}\zeta_{L,p}(s)
\end{equation}
into Euler factors $\zeta_{L,p}(s):=\sum_{i=0}^\infty b_{p^i}p^{-is}$,
enumerating subrings of finite $p$-power index. It is worth pointing
out that, for each prime $p$, the Euler factor $\zeta_{L,p}(s)$ is the
zeta function of the $\Zp$-algebra $L_p:=L\otimes\Zp$, where $\Zp$ is
the ring of $p$-adic integers.

In the study of nilpotent groups, nilpotent Lie rings play an
important role (which motivated our choice of notation `$L$' for a
general ring). A \emph{Lie ring} is a finitely generated abelian group
with a bi-additive product $[\,,]$ (called `Lie-bracket') satisfying
the Jacobi-identity
$$\forall x,y,z\in L:\; [x,[y,z]] + [y,[z,x]] + [z,[x,y]] = 0 $$ and,
for all $x\in L$, $[x,x]=0$. The lower central series of $L$ is
defined inductively via $\gamma_1(L):=L$,
$\gamma_{i}(L):=[\gamma_{i-1}(L),L]$ for $i\geq2$. We say that a Lie
ring $L$ is nilpotent of class $c$ if $\gamma_{c+1}(L)=\{0\}$ but
$\gamma_{c}(L)\not=\{0\}$. For example, a Lie ring is nilpotent of
class~$1$ if and only if it is abelian, and nilpotent of class~$2$ if
and only if the derived ring is central, i.e.~if $L':=[L,L]\leq Z(L)$.

\begin{example} \label{example sl2}
  Let $\sltwo(\Z)$ be the Lie ring of traceless integral
  $2\times2$-matrices with Lie bracket $[x,y]:=xy-yx$. It has a
  $\Z$-basis consisting of the matrices
$$e:=\left(\begin{matrix}0&1\\0&0\end{matrix}\right),\quad
f:=\left(\begin{matrix}0&0\\1&0\end{matrix}\right), \quad
h:=\left(\begin{matrix}1&0\\0&-1\end{matrix}\right) $$ which satisfy
the relations $[h,e]=2e$, $[h,f]=-2f$, $[e,f]=h$. A non-trivial
computation shows that, for odd prime $p$,
$$\zeta_{\sltwo(\Z),p}(s)=\zeta_{\sltwo(\Zp)}(s)=\zeta_p(s)\zeta_p(s-1)\zeta_p(2s-1)\zeta_p(2s-2)\zeta_p(3s-1)^{-1}$$
whereas, for $p=2$,
$$\zeta_{\sltwo(\Z),2}(s)=\zeta_{\sltwo(\Z_2)}(s)=\zeta_2(s)\zeta_2(s-1)\zeta_2(2s-1)\zeta_2(2s-2)(1+6\cdot
2^{-2s}-8\cdot 2^{-3s}).$$ This was first proved
in~\cite{duSTaylor/02}. We will sketch an alternative proof in
Section~\ref{section 3D}.
\end{example}

\subsection{Linearisation}\label{section linearisation} 
Whilst it is possible to analyse the Euler factors of zeta functions
of nilpotent groups directly (\cite[Section 2]{GSS/88}), it is often
useful to exploit the fact that the study of subgroup growth of
nilpotent groups can be \emph{linearised}, i.e.~reduced to the study
of subring growth of suitable (nilpotent Lie) rings associated with
these groups. Let $G$ be a $\T$-group. The Malcev correspondence
assigns to $G$ a $\Q$-Lie algebra $\mathcal{L}=\mathcal{L}(G)$ which
contains a Lie subring $L=L(G)$. The dimension of $\mathcal{L}$ as a
$\Q$-vector space (and thus the torsion-free rank of $L$ as a
$\Z$-module) coincides with the \emph{Hirsch length} $h(G)$ of $G$,
the number of infinite cyclic factors in a polycyclic series for
$G$. It can also be shown that $L$ is nilpotent of class $c$, where
$c$ is the nilpotency class of $G$. It has the property that, for
almost all (i.e.~all but finitely many) primes $p$,
\begin{equation}\label{equation malcev subgroups}
 \zeta_{G,p}(s)=\zeta_{L,p}(s)
\end{equation} 
(see \cite[Section 4]{GSS/88} for details). 

The exclusion of a finite number of primes is a recurrent phenomenon
in the theory of zeta functions of nilpotent groups and of rings.

If $G$ is nilpotent of class $1$, i.e.~abelian, there is of course
nothing to do: we choose $(L,+)=(G,\cdot)$, with trivial ring
structure. If $G$ is nilpotent of class $2$, i.e.~if $G'\leq Z(G)$, we
may choose
\begin{equation}\label{definition lie ring class 2}
L:=Z(G)\oplus G/Z(G),
\end{equation}
with Lie bracket induced from taking commutators in the groups. It
satisfies the identities~\eqref{equation malcev subgroups} for all
primes, i.e.~$\zeta_{G}(s)=\zeta_{L}(s)$ in this case.

We illustrate the passage from nilpotent groups to nilpotent Lie rings
with an important and prototypical example and some of its
generalisations.

\begin{example}\label{example heisenberg}
The group $$G:=\left(\begin{matrix} 1&\Z&\Z\\0&1&\Z\\0&0&1
\end{matrix}\right)$$ is called the discrete Heisenberg group of
$3\times3$-upper-unitriangular matrices over the integers. It can
easily be seen to be nilpotent of class $2$ and of Hirsch length
$3$. In fact, its centre $Z(G)$ coincides with the derived group~$G'$,
which is the infinite cyclic subgroup generated by the matrix
$$\left(\begin{matrix}1&0&1\\0&1&0\\0&0&1\end{matrix}\right).$$ It is
not hard to see that the Lie ring $L$ constructed in~\eqref{definition
lie ring class 2} has a presentation
$$L=\langle x,y,z|\;[x,y]=z, [x,z]=[y,z]=e\rangle.$$ It can be shown
that
\begin{equation}\label{zeta heisenberg}
\zeta_{G}(s)=\zeta_{L}(s)=\zeta(s)\zeta(s-1)\zeta(2s-3)\zeta(2s-2)\zeta(3s-3)^{-1}.
\end{equation}
This was first proved in~\cite{GSS/88}. We will prove this in
Proposition~\ref{proposition heisenberg} and sketch another proof in
Section~\ref{section 3D}.
\end{example}

\begin{example}\label{example F23}
  The Heisenberg group has many aspects that may be generalised. For
  instance, it is the free nilpotent group of nilpotency class $2$ on
  two generators. In general, given integers $c,d\geq2$, the free
  nilpotent group $F_{c,d}$ on $d$ generators and nilpotency class $c$
  may be defined as the quotient
  $$F_{c,d}:=F_d/\gamma_{c+1}(F_d)$$ of the free group $F_d$ on $d$
  letters by the $c+1$-th term of its lower central series. The groups
  $F_{2,d}$, for example, have a presentation
  $$F_{2,d}=\langle x_1,\dots,x_d,y_{11},y_{12},\dots,y_{d-1\,
  d}|\;[x_i,x_j]=y_{ij}, \text{ all other $[\,,]$ trivial}\rangle.$$
  The associated Lie rings $L_{2,d}$ have identical presentations.
  
Computing explicit formulae for zeta functions of groups is in general
  very difficult, even if the groups have quite a transparent
  structure. For the zeta functions $\zeta_{F_{c,d}}(s)$, explicit
  formulae are only known for the cases $(1,d)$
  (cf.~Example~\ref{example abelian}) and
  $(c,d)\in\{(2,2),(2,3),(3,2)\}$. For example
  (cf.~\cite[2.7.1]{duSautoyWoodward/08}),
\begin{align*}\zeta_{F_{2,3}}(s)=&\zeta_{\Z^3}(s)\zeta(2s-4)\zeta(2s-5)\zeta(2s-6)\zeta(3s-6)\cdot\\
&\zeta(3s-7)\zeta(3s-8)\zeta(4s-8)^{-1}\prod_{p \text{
prime}}W_{2,3}(p,p^{-s}),
\end{align*}
where 
\begin{align*}
W_{2,3}(X,Y)=&1+X^3Y^2+X^4Y^2+X^5Y^2-X^4Y^3-X^5Y^3\\&-X^6Y^3-X^7Y^4-X^9Y^4-X^{10}Y^5-X^{11}Y^5\\&-X^{12}Y^5+X^{11}Y^6+X^{12}Y^6+X^{13}Y^6+X^{16}Y^8.
\end{align*}
\end{example}

We have seen that the theory of zeta functions of nilpotent groups
can, to a great extent, be reduced to the study of the zeta functions
of nilpotent Lie rings. It is worth recalling, however, that the
theory of zeta functions of rings we are about to present applies to
much more general rings.

\bigskip
\subsection{Layout of the paper} 
In Section~\ref{section local} we study local and global aspects of
subring zeta functions of rings, reviewing some of the methods
available to study these functions. By the linearisation results
outlined in the introduction, these yield, as corollaries, theorems
about (generic local) zeta functions of $\T$-groups. We put particular
emphasis on connections with the theory of linear homogeneous
diophantine equations and on local functional equations.

Some of the manifold generalisations and variations of the concept of
the zeta function of a group or ring are reviewed in
Section~\ref{section variations}. We concentrate on ideal (or normal)
zeta functions of rings (or nilpotent groups, respectively) and
representation zeta functions of nilpotent, arithmetic and $p$-adic
analytic groups.

In Section~\ref{section open problems} we present a collection of what
we believe are major open questions and conjectures in the area.

\bigskip
We use the following notation.
\medskip

\begin{tabular}{l|l}
$\N$ & the set $\{1,2,\dots\}$ of natural numbers
  \\ $I=\{i_1,\dots,i_l\}_<$ & the finite set of natural numbers
  $i_1<\dots<i_l$\\ $I_0$ & the set $I\cup\{0\}$ for $I\subseteq\N$ \\
$[k]$ & the set $\{1,\dots,k\}$, $k\in\N$ \\ 
$S_n$ & the symmetric group on $n$ letters\\ $M^t$&
the transpose of a matrix $M$\\ 
$v_p$ & the $p$-adic valuation ($p$ a
prime)\\

$\Zp$ & the ring of $p$-adic integers\\
$\Qp$ & the field of $p$-adic numbers\\
$|x|_p$& the $p$-adic absolute value of a $p$-adic number,\\&defined by $|x|_p:=p^{-v_p(x)}$\\
$\|\mathcal{S}\|_p$& the $p$-adic absolute value of a set $\mathcal{S}$ of $p$-adic\\&numbers, defined by $\|\mcS\|_p:=\max\{|s|_p|\;s\in\mcS\}$\\
$[\Lambda]$& the homothety class $\Qp^*\Lambda$ of a (full) lattice
$\Lambda$ in $\Qp^n$\\
$\delta_P$ & the `Kronecker delta' which is equal to $1$ if\\ & the
property $P$ holds and equal to $0$ otherwise
\end{tabular}
\medskip

\noindent Given a set $\bff$ of polynomials and a polynomial $g$, we
write $g\bff$ for $\{gf|\,f\in\bff\}$, and $(\bff)$ for the ideal
generated by~$\bff$.

\section{Local and global zeta functions of groups and rings}\label{section local}
Let $L$ be a ring. Given equation~\eqref{euler rings}, the problem of
studying the zeta function $\zeta_L(s)$ is reduced to the problem of
understanding the Euler factors $\zeta_{L,p}(s)$, $p$ prime, and the
analytic properties of their Euler product. The following are natural
questions:
\begin{enumerate}
\item What do the local factors $\zeta_{L,p}(s)$ have in common?
What is their structure?
\item How do the local factors vary with the prime $p$?
\end{enumerate}
In the following subsections we will explore some of the existing
methods to analyse local zeta functions of rings, and will address
both of these questions.

\subsection{Rationality and variation with the prime}\label{subsection rat and uni}
In all the examples we have seen, the local factors all shared a
number of features. In particular, they were all rational functions in
the parameter $p^{-s}$. This is no coincidence:

\begin{theorem}\cite[Theorem 3.5]{GSS/88}\label{theorem
rationality} For all primes $p$, the local zeta function
$\zeta_{L,p}(s)$ is a rational function in $p^{-s}$, i.e.~there is a
rational function $W_p(Y)=P_p(Y)/Q_p(Y)\in\Q(Y)$ such that
$$W_p(p^{-s})=\zeta_{L,p}(s).$$
\end{theorem}

The proof of this theorem uses deep results from the theory of
$p$-adic integration, which we survey to some degree below.

Theorem~\ref{theorem rationality} asserts that the sequence
$(b_{p^i}(L))$ of the numbers of subrings of $L$ of index $p^i$
satisfies a strong regularity property: it is easy to see that a
generating function of the form $\sum_{i=0}^\infty b_{p^i}t^i$ is
rational in the variable~$t$ if and only if there is a finite linear
recurrence relation on the coefficients $b_{p^i}$, the length of which
is determined by the degree of the denominator (cf.~\cite[Theorem
  4.1.1]{Stanley/97}). In other words, the numbers of finite index
subalgebras of $L_p$ are already determined by the numbers of
subalgebras in some finite quotient of~$L_p$.

\emph{A priori}, Theorem~\ref{theorem rationality} does not give us
any information on the shape of the rational functions $W_p$. In
particular, it does not tell us how the lengths of these recurrence
relations depend on the prime, or when they set in. In the examples
above we observe that the denominators are all of the form
$\prod_{i\in I}(1-p^{a_i-b_is})$ for suitable non-negative integers
$a_i,b_i$. This, too, is a general phenomenon.

\begin{theorem}\cite{duS/94}\label{theorem duS denominators}
For each $n\in\N$ there exists a finite index set $I_n$, and finitely
    many pairs $(a_i,b_i)_{i\in I_n}$ of natural numbers such that, if
    $L$ is a ring of additive rank $n$, for all primes $p$ the
    denominator polynomial $Q_p(Y)\in\Q[Y]$ in Theorem~\ref{theorem
    rationality} can be taken to divide $\prod_{i\in
    I_n}(1-p^{a_i}Y^{b_j})$.
\end{theorem}

Theorem~\ref{theorem duS denominators} implies that the degrees in
$p^{-s}$ of the denominator polynomials $Q_p(Y)$ are bounded when $L$
ranges over all rings of a given rank $n$. In particular, there is a
uniform upper bound on the lengths of the recurrence relations
satisfied by the sequences $(b_{p^i}(L))_i$ for fixed $L$ as~$p$
ranges over the primes. It also shows that the coefficients of
$Q_p(Y)$ are polynomials in~$p$, so that the denominators of the Euler
factors are really polynomials in $p$ and $p^{-s}$. The proof of
Theorem~\ref{theorem duS denominators} relies on non-constructive
methods from model theory. No procedure is known to describe
explicitly (even just a reasonably small superset of) the factors of
the denominator of the local zeta functions of a given ring.

The numerators of the Euler factors have, in general, a far more
complicated and interesting structure. In all of the examples we have
encountered so far, the coefficients of the polynomials $P_p(Y)$, too,
were -- at least for almost all primes~$p$ -- polynomials in $p$. It
was known already to the authors of~\cite{GSS/88} that this is not a
general feature. Their paper contains examples of zeta functions of
nilpotent groups whose local factor at the prime $p$ depends on how
the rational prime $p$ behaves in a number field. The right framework
to explain this phenomenon, however, was not discovered until much
later.

\begin{theorem}\cite[Theorem 1.3]{duSG/00}\label{theorem duSG}
Let $L$ be a ring. There are smooth algebraic varieties $V_t$,
$t\in[m]$, defined over $\Q$, and rational functions
$W_t(X,Y)\in\Q(X,Y)$ such that, for almost all primes $p$,

\begin{equation}\label{formula local}
\zeta_{L,p}(s)=\sum_{t=1}^m c_t(p)W_t(p,p^{-s}),
\end{equation}
where $c_t(p)$ denotes the number of $\Fp$-rational
points\footnote{The formulation given here follows from the original
formulation in~\cite{duSG/00} by the inclusion-exclusion principle.} of
$\ol{V_t}$, the reduction modulo $p$ of $V_t$.
\end{theorem}

We will remark on the proof of this theorem at the end of
Section~\ref{subsection lhde}.

In general, the numbers of $\Fp$-rational points of the
reduction modulo $p$ of varieties defined over $\Q$ will not be
polynomials in $p$, as the following example shows.

\begin{example}\label{example ec}
Let $E$ be the elliptic curve defined by the equation $y^2=x^3-x$. For
a prime $p$ we denote by $c(p)$ the number of $\Fp$-rational points of
$\ol{V_t}$, the reduction modulo $p$ of $E$, i.e.
$$c(p):=|\{(x,y)\in\Fp^2|\;y^2=x^3-x\}|.$$ It is known (\cite[\S
  18.4]{IrelandRosen/82}\footnote{The discrepancy with the formula
  given in \cite[\S 18.4, Theorem 5]{IrelandRosen/82} comes from the
  fact that there $c(p)$ refers to the number of \emph{projective}
  points of $E$, which includes also a point at infinity
  (cf.~Example~\ref{example ec projective}). This should also have
  been taken into account in~\cite[Example 1]{duSSegal/00}.})  that,
if $p\equiv3\mod(4)$, then $c(p)=p$. If, however, $p\equiv1\mod(4)$,
then $c(p)=p-(\pi+\ol{\pi})$, where $\pi$ is the complex number
satisfying $p=\pi{\ol{\pi}}$ and $\pi\equiv1\mod(2+2i)$.
\end{example}

It is not clear \emph{a priori} that varieties with such `wild'
arithmetical behaviour can occur in the description of zeta functions
of rings given in~\eqref{formula
  local}. In~\cite{duS-ecI/01,duS-ecII/01} du Sautoy gave an example
of a class-$2$-nilpotent Lie ring (or, equivalently,
class-$2$-nilpotent group) whose local zeta functions involve the
cardinalities $c(p)$ associated with the elliptic curve in Example
\ref{example ec}. In particular, he proved that the zeta function of
this Lie ring is not `finitely uniform'. We say that $\zeta_{L}(s)$ is
\emph{finitely uniform} if there are finitely many rational functions
$W_i(X,Y)\in\Q(X,Y)$, $i\in I$, a finite index set, such that for
every prime $p$ there exists an $i=i(p)$ such that
$\zeta_{L,p}(s)=W_i(p,p^{-s})$. We say that $\zeta_{L}(s)$ is
\emph{uniform} if it is finitely uniform for $|I|=1$ and \emph{almost
  uniform} if there exists a rational function $W(X,Y)$ such that
$\zeta_{L,p}(s)=W(p,p^{-s})$ for almost all~$p$. We will revisit du
Sautoy's example in Section~\ref{section variations}, where we will
look at the \emph{ideal} zeta function of this particular Lie ring,
counting only ideals of finite index. For this variant, we will be
able to give an explicit formula for the local zeta functions,
illustrating Theorem~\ref{theorem duSG analytic} (or rather its
analogue for ideal zeta functions of rings) in this particular
case. It seems worth pointing out, however, that so far all the zeta
functions $\zeta_{L}(s)$ of rings~$L$ for which explicit formulae are
known are finitely uniform.

For future reference we study in some detail an important sample
family of varieties with very `uniform' reduction behaviour modulo
$p$. They play a key role in explicit formulae for zeta functions of
rings.

\subsection{Flag varieties and Coxeter groups}\label{subsection flags}

Let $V$ denote an $n$-dimensional vector space over a field~$k$. For
each $i\in[n-1]$, the set $G_{n,i}(k)$ of subspaces of $V$ of
dimension $i$ can be given the structure of a smooth projective
variety over $k$, called the $i$-th \emph{Grassmannian} of $V$. Given
a prime power $q$ we obtain $G_{n,i}(\Fq)$. We define, for $1\leq i<n$, the
polynomial
$$\binom{n}{i}_X:=\prod_{j = 0}^{i-1}
(X^{n-j}-1)/(X^{i-j}-1)\in\Z[X].$$ It is not hard to prove that
$|G_{n,i}(\Fq)|=\binom{n}{i}_q\in\Z[q]$. For example, the cardinality
$|\mathbb{P}^{n-1}(\Fq)|$ of the $n-1$-dimensional projective space of
lines in~$\Fq^n$ is given by
$\binom{n}{1}_q=(q^n-1)/(q-1)=1+q+\dots+q^{n-1}$.

More generally, let $I=\{i_1,\dots,i_l\}_<$, be a subset of $[n-1]$. A
\emph{flag of type~$I$} in $V$ is a sequence $(V_i)_{i\in I}$ of
subspaces of $V$ satisfying
$$\{0\}\subsetneq V_{i_1}\subsetneq V_{i_2}\subsetneq\cdots\subsetneq
V_{i_l}\subsetneq V$$ and, for all $i\in I$, $\dim(V_i)=i$. A flag is
called \emph{complete} if it is of type $I=[n-1]$. The set of flags of
type $I$ can be given the structure of a smooth projective variety
over $k$. If $k=\Fq$, we obtain the variety of flags of type $I$ in
$\Fq^n$. We define the polynomial
\begin{equation}\label{definition flag polynomials}
\binom{n}{I}_X:=\binom{n}{i_l}_X\binom{i_l}{i_{l-1}}_X\dots\binom{i_2}{i_1}_X\in\Z[X].
\end{equation}
The numbers $\binom{n}{I}_q$ are called \emph{$q$-binomial
coefficients} or \emph{Gaussian polynomials}. One easily proves
inductively that the number of flags of type $I$ in $\Fq^n$ is given
by the polynomial $\binom{n}{I}_q\in\Z[q]$. For example, the number of
complete flags in $\Fq^3$ is given by $(1+q+q^2)(1+q)=1+2q+2q^2+q^3$.

For further applications we shall need an expression for the
polynomials~$\binom{n}{I}_{X}$ in terms of Coxeter group theoretic
notions.

\begin{definition}
Let $S_n$ be the symmetric group of $n$ letters with standard
(Coxeter) generators $s_1,\dots,s_{n-1}$ (in cycle notation these are
the transpositions $s_i=(i\;i+1)$). Let $w\in S_n$. The
\emph{(Coxeter) length} $\ell(w)$ is the length of a shortest word in
the generators $s_i$ representing $w$. The \emph{(left) descent type}
$D_L(w)$ is the set $\{i\in[n-1]|\;w(i+1)<w(i)\}$.
\end{definition}
\noindent It can be shown that
\begin{equation}\label{descent characterisation}
D_L(w)=\{i\in[n-1]|\;\ell(s_iw)<\ell(w)\}.
\end{equation}

\begin{proposition}
\label{proposition binomial length} 
Let $q$ be a prime power. For all $I\subseteq[n-1]$
$$\binom{n}{I}_q=\sum_{w\in S_n, \,D_L(w)\subseteq I} q^{\ell(w)}.$$
\end{proposition}

\begin{proof}
We first prove the proposition for $I=[n-1]$. In this case,
$\binom{n}{[n-1]}_q$ gives the number of complete flags
$(V_i)_{i\in[n-1]}$ in the finite vector space $\Fq^n$. These may be
also viewed as the cosets $\GL_n(\Fq)/B(\Fq)$, where $B$ denotes the
Borel subgroup of upper-triangular matrices in $\GL_n$. It is
well-known that the algebraic group $\GL_n$ satisfies a Bruhat
decomposition
$$\GL_n=\bigcup_{w\in S_n} B w B $$ (where we identify permutations
in $S_n$ with permutation matrices in $\GL_n$, acting from the left on
unit column vectors, say).  Therefore
$$\GL_n(\Fq)/B(\Fq)=\bigcup_{w\in S_n} B(\Fq) w B(\Fq)/B(\Fq).$$ The
disjoint pieces $\Omega_w(\Fq):=B(\Fq) w B(\Fq)/B(\Fq)$, $w\in S_n$,
are called \emph{Schubert cells}. It can be shown that each Schubert
cell $\Omega_w(\Fq)$ is an affine space over $\Fq$ of dimension given
by the length $\ell(w)$. Indeed, a complete set of representatives of
$B(\Fq) w B(\Fq)/B(\Fq)$, of size $q^{\ell(w)}$, is obtained in the
following way: start with the permutation matrix corresponding to
$w$. Substitute an arbitrary entry in $\Fq$ for each of the zeros of
this matrix which is not positioned anywhere below or to the right of
a~$1$. We conclude that
\begin{align*}
\binom{n}{[n-1]}_q&=|\GL_n(\Fq)/B(\Fq)|=\left|\bigcup_{w\in S_n}  B(\Fq) w B(\Fq)/B(\Fq)\right|\\
&=\sum_{w\in S_n}|\Omega_w(\Fq)|=\sum_{w\in S_n} q^{\dim(\Omega_w)}=\sum_{w\in S_n}q^{\ell(w)}
\end{align*}
This proves the proposition in the special case $I=[n-1]$.

\begin{example}
Let $n=5$. The Schubert cell $\Omega_w$ indexed by the element
$$w=(1532)\in S_5$$ may be identified with the set of matrices of the
form
$$\left(\begin{matrix} *&*&1&0&0\\ *&*&0&*&1\\ *&1&0&0&0\\
  *&0&0&1&0\\1&0&0&0&0\end{matrix}\right),$$ where $*$ may take any
  value in $\Fq$. Note that there are $7$ $*$'s, reflecting the fact
  that $\ell(w)=7$. Indeed, a shortest word representing $w$ is
  $$s_2s_3s_1s_4s_3s_2s_1.$$ The descent type of $w$ is
  $D_L(w)=\{2,4\}$.
\end{example}

In the general case, given $I=\{i_1,\dots,i_l\}_<\subseteq[n-1]$,
$\binom{n}{I}_q$ is the number of flags $(V_i)_{i\in I}$,
$\dim(V_{i})=i$, in $\Fq^n$. These are in $1-1$-correspondence with
cosets $\GL_n(\Fq)/B_I(\Fq)$, where $B_I(\Fq)$ is the parabolic
subgroup consisting of matrices of the form
$$\left(\begin{matrix}
  \gamma_{i_1}&*&*&*\\0&\gamma_{i_2-i_1}&*&*\\\vdots&\ddots&\ddots&\vdots\\0&0&0&\gamma_{n-i_l}\end{matrix}\right),$$
  where $\gamma_i\in\GL_i(\Fq)$.
  
  Among the Schubert cells $\Omega_{w}(\Fq)$ which are being
  identified by passing to cosets of $B_I(\Fq)$ there is a unique one
  with minimal dimension. It is not hard to see that these cells are
  exactly the cells indexed by elements $w\in S_n$ with
  $D_L(w)\subseteq I$, and that they constitute a set of
  representatives for the cosets $\GL_n(\Fq)/B_I(\Fq)$. We obtain
\begin{align*}
\binom{n}{I}_{q}&=|\GL_n(\Fq)/B_I(\Fq)|=\left|\bigcup _{w\in
S_n}B(\Fq)wB(\Fq)/B_I(\Fq)\right|\\ &=\sum_{w\in S_n, D_L(w)\subseteq
I}|\Omega_w(\Fq)|=\sum_{w\in S_n, D_L(w)\subseteq I}q^{\ell(w)}.
\end{align*}
This proves Proposition~\ref{proposition binomial length} in general.
\end{proof}

\subsection{Counting with $p$-adic integrals}\label{subsection counting}
The idea to employ tools from the theory of $p$-adic integration to
count subgroups and subrings is as old as the subject. It was first
put to work in~\cite{GSS/88}, and was further developed
in~\cite{duSG/00} and ~\cite{Voll/06a}. All of these $p$-adic
integrals are in some sense generalisations of Igusa's local zeta
function, which we describe first. This will allow us to give a first
proof of formula~\eqref{formula abelian} for the zeta functions of
abelian groups.  We will also show how a formulation in terms of
$p$-adic integrals enables us to express the local zeta functions of
the Heisenberg Lie ring (cf.~Example~\ref{example heisenberg}) in
terms of the generating function associated with a polyhedral cone (or,
equivalently, a system of linear homogeneous diophantine equations),
which we may evaluate to confirm formula~\eqref{zeta heisenberg}. We
will study these in some detail in Section~\ref{subsection lhde}.

The $p$-adic integrals we consider are all variants of Igusa's local
zeta function. Given a polynomial $f\in\Z[x_1,\dots,x_n]$, Igusa's
local zeta function associated with $f$ is the $p$-adic integral
$$Z_f(s):=\int_{\Zp^n}|f(\bfx)|_p^{s}d\mu^{(n)}.$$ Here, $\Zp$ are the
$p$-adic integers, $\mu^{(n)}$ is the (additive) Haar measure on $\Zp^n$
(normalised such that $\mu^{(n)}(\Zp^n)=1$), $|\;\;|_p$ denotes the $p$-adic
norm (defined by $|a|_p:=p^{-v_p(a)}$, where $v_p(a)=r$ if $a=p^rb$
with $p\nmid b$), and $s$ is a complex variable. (For a reminder
about the Haar measure on $\Zp^n$,
see~\cite[Section~1.6]{duS-ennui/02}.) 

Igusa's local zeta function associated with the polynomial $f$ is a good
tool to understand the sequence $(N_m)$, where $N_m$ denotes the
number of solutions of the congruence $f(\bfx)\equiv0\mod(p^m)$. These
numbers may be encoded in a Poincar\'e series
$$P_f(t):=\sum_{m=0}^\infty p^{-nm}N_mt^m.$$
This Poincar\'e series is related to the $p$-adic integral via the formula
\begin{equation}\label{equation igusa}
P_f(p^{-s})=\frac{1-p^{-s}Z_f(s)}{1-p^{-s}}.
\end{equation}
Indeed, $p^{-nm}N_m$ is the measure of the set
$\{\bfx\in\Zp^n|\;v_p(f(\bfx))\geq m\}$ and thus 
$$\mu^{(n)}(\{\bfx\in\Zp^n|\;v_p(f(\bfx))=m\})=p^{-nm}N_m-p^{-n(m+1)}N_{m+1}.$$
Thus
\begin{align*}
Z_f(s)&=\sum_{m=0}^\infty \mu^{(n)}(\{\bfx\in\Zp^n|\;v_p(f(\bfx))=m\})p^{-sm}\\
&=\sum_{m=0}^\infty \left(p^{-nm}N_m-p^{-n(m+1)}N_{m+1}\right)p^{-sm}\\
&=P_f(p^{-s})-p^s\left(P_f(p^{-s})-1\right)\\
&=(1-p^s)P_f(p^{-s})+p^s,
\end{align*}
which is equivalent to~\eqref{equation igusa}. As an example of the
above formula, consider the integral
$$Z(s):=\int_{\Zp}|x|_p^sd\mu^{(1)}.$$ Observing that the associated
Poincar\'e series equals 
$$P(p^{-s})=\sum_{m=0}^\infty
(p^{-1-s})^m=\zeta_p(s+1)=\frac{1}{1-p^{-1-s}}$$ we deduce that
\begin{equation}\label{equation igusa x}
Z(s)=\frac{1-p^{-1}}{1-p^{-1-s}}=(1-p^{-1 })\zeta_p(s+1).
\end{equation}

We now explore how $p$-adic integrals may be used to count subgroups,
by giving a first proof of formula~\eqref{formula abelian}. Recall
that we may consider $\Z^n$ as a ring with trivial multiplication, so
counting subgroups and counting subrings is the same thing in this
case. It suffices to prove that, for each prime~$p$,
\begin{equation}\label{formula abelian local}
\zeta_{\Z^n,p}(s)=\zeta_{\Zp^n}(s)=\zeta_p(s)\zeta_p(s-1)\cdots\zeta_p(s-(n-1)).
\end{equation}

The first equation is clear. For the second equation, assume that
$\Zp^n=\Zp e_1\oplus\dots\oplus\Zp e_n$ as $\Zp$-module, and set
$\G:=\text{GL}_n(\Zp)$. Then subgroups of $\Zp^n$ of finite index may
be identified with right $\G$-cosets of $n\times n$-matrices over
$\Zp$ with non-zero determinant. Indeed, every such subgroup may be
generated by $n$ generators, whose coordinates with respect to the
chosen basis may be encoded in the rows of an $n\times n$-matrix
over~$\Zp$. Two such matrices $M_1$ and $M_2$ correspond to the same
subgroup if and only if there is an element $\gamma\in\G$ such that
$M_1=\gamma M_2$. In fact, one sees easily that these matrices may be
chosen to lie in the set $\trnzp$ of upper-triangular matrices over
$\Zp$, so that subgroups $H$ correspond to cosets $\mcU M$, where
$M\in\trnzp$ and $\mcU:=\Gamma\cap\trnzp$.

Now choose, for each $H\leq_f\Zp^n$, a representative $M_H$ in $\mcU
M=:\mcM(H)$, the $\mcU$-coset in $\trnzp$ corresponding to $H$. Notice
that
\begin{equation}\label{equation index=determinant}
|\Zp^n:H|=|\det(M_H)|_p^{-1},
\end{equation}
 and that
\begin{equation}\label{equation measure}
\mu(\mcM(H))=(1-p^{-1})^n\prod_{i=1}^n|(M_H)_{ii}|_p^{i}
\end{equation}
where
$\mu$ denotes the additive Haar measure on
$\trnzp\cong\Zp^{\binom{n+1}{2}}$, normalised so that $\mu(\trnzp)=1$.
We thus obtain a partition 
\begin{equation}\label{equation disjoint union}
\trnzp=\bigcup_{H\leq_f\Zp^n}\mcM(H)\cup\text{Tr}^0(n,\Zp),
\end{equation}
(where $\text{Tr}^0(n,\Zp)$ denotes the set of $n\times
n$-upper-triangular matrices over $\Zp$ with zero determinant, of
Haar-measure zero) and compute
\begin{align*}
\sum_{H\leq\Zp^n}|\Zp^n:H|^{-s}&=\sum_H
  |\det(M_H)|_p^s&&\eqref{equation index=determinant}\\ &=\sum_H
  \mu(\mathcal{M}(H))^{-1}\mu(\mathcal{M}(H))\prod_{i=1}^n|(M_H)_{ii}|_p^s\\
  &=\sum_H(1-p^{-1})^{-n}\prod_{i=1}^n|(M_H)_{ii}|_p^{-i}\int_{\mathcal{M}(H)}\prod_{i=1}^n|(M_H)_{ii}|_p^s
  d\mu&&\eqref{equation measure}\\
  &=(1-p^{-1})^{-n}\int_{\trnzp}\prod_{i=1}^n|M_{ii}|_p^{s-i}
  d\mu&&\eqref{equation disjoint union}\\
  &=(1-p^{-1})^{-n}\int_{\Zp^n}\prod_{i=1}^n |x_i|_p^{s-i}d\mu^{(n)}\\
  &=(1-p^{-1})^{-n}\prod_{i=1}^n\int_{\Zp}|x|_p^{s-i}d\mu^{(1)}&&\text{Fubini}\\
  &=\prod_{i=1}^n\zeta_p(s-(i-1)),&&\eqref{equation igusa x}
\end{align*}
which proves~\eqref{formula abelian local}.

Note that we managed to compute each of the local factors of
$\zeta_{\Z^n}(s)$ by expressing it as an integral over the affine
space $\trnzp$ of upper-triangular matrices. The integrand in this
integral is a simple function of the diagonal entries of the
matrices. How does this approach vary if we consider rings with
nontrivial multiplication? For arbitrary rings $L$, our above analysis
carries through up to (and including) equation~\eqref{equation
measure}. In general, however, not every coset $\mcU M$ will
correspond to a subring of $L$. We therefore need to describe
conditions for such a coset to define a subring.

Let us return to Example~\ref{example heisenberg} of the discrete
Heisenberg group. Its associated Lie ring $L$ has a $\Z$-basis
$(x,y,z)$, where $[x,y]=z$ is the only non-trivial relation. To
compute its local zeta function at the prime $p$, we need to count
subalgebras in the $\Zp$-algebra $L_p:=\Zp\otimes L$. The rows of a
matrix $M=(M_{ij})\in\trthreezp$ encode the generators of a full
additive sublattice of $\Zp^3$. To determine whether such a matrix gives
rise to a subalgebra we need to check whether this sublattice
is closed under taking Lie brackets of its generators. In this case it
is easy to see that the only condition we need to check is
\begin{equation*}\label{heisenberg condition}
[M_{11}x+M_{12}y+M_{13}z,M_{22}y+M_{23}z]\in \langle M_{33}z\rangle_{\Zp}.
\end{equation*}
Using the commutator relation $[x,y]=z$ and the bilinearity of the Lie
bracket~$[\,,]$, we see that this condition is equivalent to
\begin{equation}\label{div condition heisenberg}
M_{33}\mid M_{11}M_{22}.
\end{equation}
Note that this divisibility condition is equivalent to the inequality
of $p$-adic valuations
$$ v_p(M_{33})\leq v_p(M_{11})+v_p(M_{22}).$$ We thus obtain
\begin{align*}
\sum_{H\leq
L_p}|L_p:H|^{-s}&=(1-p^{-1})^{-3}\int_{\left\{\substack{M\in\trthreezp|\\\;M_{33}\mid
M_{11}M_{22}}\right\}}\prod_{i=1}^3|M_{ii}|_p^{s-i}d\mu^{(6)}\\&=(1-p^{-1})^{-3}\int_{\{\bfx\in\Zp^3|\;x_3\mid x_1x_2\}}|x_{1}|_p^{s-1}|x_{2}|_p^{s-2}|x_{3}|_p^{s-3}d\mu^{(3)}\\&=\sum_{\{\bfm\in\N_0^3|\;m_3\leq
m_1+m_2\}}(p^{-s})^{m_1}(p^{1-s})^{m_2}(p^{2-s})^{m_3}\label{sum heisenberg}.
\end{align*}

It is not hard to compute this sum explicitly (see
Proposition~\ref{proposition heisenberg} below). It is useful,
however, to observe that it may be interpreted as a generating
function associated with a system of linear homogeneous diophantine
equations.

\subsection{Linear homogeneous diophantine equations} \label{subsection lhde}
 Let $\Phi$ be an $r\times m$ matrix over $\Z$ (without loss of
generality of rank $r$), and consider the system of linear equations
\begin{equation}\label{lhde}
\Phi\bsalpha={\bf 0},
\end{equation}
where $\bsalpha^t=(\alpha_1,\dots,\alpha_m)$ and ${\bf
0}\in\N_0^r$. The set of \emph{non-negative} integral solutions of
\eqref{lhde} form a commutative monoid
$\mcE:=\{\bsalpha\in\N_0^m|\;\Phi\bsalpha={\bf 0}\}$ with identity
under addition. One approach to study this monoid is to investigate
the generating function
$$E(\bfX):=E_\Phi(\bfX):=\sum_{\bsalpha\in
\mcE}\bfX^{\bsalpha},$$ where
$\bfX^{\bsalpha}=X_1^{\alpha_1}\dots X_m^{\alpha_m}$ is a
monomial in variables $X_1,\dots,X_m$.

The generating functions $E_\Phi(\bfX)$ have been intensely studied by
Stanley and others (\cite[Chapter 4.6]{Stanley/97}, \cite[Chapter
I]{Stanley/96}). It can be proved, for example, that $E_\Phi(\bfX)$ is
always a rational function in the variables $X_1,\dots,X_m$, with
denominator of the form $\prod_{{\bfbeta}\in CF(E)}(1-\bfX^{\bfbeta})$, where
$\beta$ ranges over the finite set $CF(E)$ of \emph{completely
fundamental solutions} to $\Phi$. (A solution $\bsbeta$ to
\eqref{lhde} is called \emph{fundamental} if, whenever
${\bsbeta}=\bsgamma+\bsdelta$ for $\bsgamma,\bsdelta\in \mcE$,
$\bsgamma=\bsdelta$ or $\bsdelta=\bsbeta$. A solution ${\bsbeta}$ to
\eqref{lhde} is called \emph{completely fundamental} if, whenever
$n{\bsbeta}=\bsgamma +\bsdelta$ for $\bsgamma,\bsdelta\in \mcE$, then
$\bsgamma=n_1{\bsbeta}$ for some $0\leq n_1\leq n$.)

We will also consider the closely related generating function
$$\ol{E}(\bfX):=\ol{E}_\Phi(\bfX):=\sum_{{\bsalpha}\in\ol{\mcE}}\bfX^{\bsalpha},$$
where $\ol{\mcE}:=\{{\bsalpha}\in\N^m|\;\Phi{\bsalpha}={\bf 0}\}$, the
semigroup of the \emph{positive} integral solutions
of~\eqref{lhde}. $\ol{E}(\bfX)$ is also a rational function in the
variables $X_1,\dots,X_m$. The following result of Stanley will be of
great importance in applications to zeta functions of rings. We denote
by $1/\bfX$ the vector of inverted variables $(1/X_1,\dots,1/X_m)$.

\begin{theorem}\cite[Theorem 4.6.14]{Stanley/97}\label{theorem stanley} 
Assume that $\ol{\mcE}\not=\varnothing$ and set $d:=\dim(\mathscr{C})$, where
$\mathscr{C}$ is the cone of non-negative \emph{real} solutions
to~\eqref{lhde}. Then 
\begin{equation}\label{stanley reciprocity}
\ol{E}(\bfX)=(-1)^d E(1/\bfX).
\end{equation}
\end{theorem}

\begin{example} \label{example trivial cone}
If $r=0$ we obtain $\mcE=\N_0^m$, with completely fundamental
solutions $\{(1,0,\dots,0),\dots,(0,\dots,0,1)\}$, yielding
$$E(\bfX)=\sum_{{\bsalpha}\in\N_0^m}\bfX^{\bsalpha}=\prod_{i=1}^m\frac{1}{1-X_i}\quad\text{
and
}\quad\ol{E}(\bfX)=\sum_{{\bsalpha}\in\N^m}\bfX^{\bsalpha}=\prod_{i=1}^m\frac{X_i}{1-X_i}.$$
\end{example}

\begin{example} \label{example cone}
Consider the matrix $\Phi=(1,1,-1,-1)$. It can be shown that the
(completely) fundamental solutions of the equation
$\alpha_1+\alpha_2-\alpha_3-\alpha_4=0$ are $(1,0,1,0)$, $(1,0,0,1)$,
$(0,1,1,0)$ and $(0,1,0,1)$. Note that there is one non-trivial
relation between these solutions:
$$(1,0,1,0)+(0,1,0,1)=(1,0,0,1)+(0,1,1,0)\;(=(1,1,1,1)).$$ This can be
used to show (see~\cite[I.11]{Stanley/96} for details) that
\begin{equation}
\label{equation gen. fun.}
E_{\Phi}(X_1,X_2,X_3,X_4)=\frac{1-X_1X_2X_3X_4}{(1-X_1X_3)(1-X_1X_4)(1-X_2X_3)(1-X_2X_4)}.
\end{equation}
\end{example}

Note that we obtain nothing more complicated if we allow inequalities
rather than equalities in~\eqref{lhde}. Indeed, an inequality may
always be expressed in terms of an equality by introducing a slack
variable. The generating functions enumerating integral points in
rational polyhedral cones (intersections of finitely many rational
half-spaces) may therefore be expressed in terms of generating
functions associated with linear homogeneous diophantine
equations. For example, $m_3\leq m_1+m_2$ if and only if there exists
$m_4\in\N_0$ such that $m_3+m_4=m_1+m_2$ or, equivalently,
$m_1+m_2-m_3-m_4=0$. We obtain the generating function enumerating
non-negative solutions of the inequality by taking the generating
function associated with the equality by setting the variable
corresponding to the slack variable to~$1$. From Example~\ref{example
  cone} we get, for instance,

\begin{multline}
\sum_{\{\bfm\in\N_0^3|\;m_3\leq
m_1+m_2\}}X_1^{m_1}X_2^{m_2}X_3^{m_3}=E_{\Phi}(X_1,X_2,X_3,1)\\=\frac{1-X_1X_2X_3}{(1-X_1X_3)(1-X_1)(1-X_2X_3)(1-X_2)}.\label{generating function heisenberg}
\end{multline}

In Section~\ref{subsection counting} we showed how the local zeta
functions of the Heisenberg Lie rings can be expressed in terms of the
rational function given in \eqref{generating function heisenberg}. We
summarise this result in

\begin{proposition}\cite[Proposition~8.1]{GSS/88}\label{proposition heisenberg}
Let $L$ be the Heisenberg Lie ring (cf.~Example~\ref{example
heisenberg}). Then, for all primes $p$, the local zeta function of $L$
equals
\begin{multline*}\zeta_{L_p}(s)=E_\Phi(p^{-s},p^{1-s},p^{2-s},1)=\\\frac{1-p^{3-3s}}{(1-p^{-s})(1-p^{1-s})(1-p^{2-2s})(1-p^{3-2s})}=\\\zeta_p(s)\zeta_p(s-1)\zeta_p(2s-2)\zeta_p(2s-3)\zeta_p(3s-3)^{-1}.
\end{multline*}
\end{proposition}

We have thus expressed the local factors of the zeta function of the
Heisenberg Lie ring in terms of the generating function associated
with a linear homogeneous equation (or, equivalently, a rational
polyhedral cone). The feasibility of this approach was a direct
consequence of the divisibility condition~\eqref{div condition
  heisenberg}. In general, things are not that simple, as the
following example shows.

\begin{example}\label{cone conditions sl2}
Let us reconsider the Lie ring $\sltwo(\Z)$ from Example~\ref{example
  sl2}. Fix a prime $p$. It is not hard (and a recommended exercise;
cf.~\cite{duSTaylor/02}) to show that the coset $\mcU M$ of a matrix
$M\in\trthreezp$ encodes the coordinates of generators of a subring of
$\sltwo(\Zp)$ if and only if
\begin{align}
v_p(M_{22})&\leq v_p(4M_{12}M_{23}),\nonumber\\
v_p(M_{22})&\leq v_p(4M_{12}M_{33})\text{ and }\nonumber\\
v_p(M_{22}M_{33})&\leq v_p(M_{11}M_{22}^2+4M_{22}M_{13}M_{23}-4M_{12}M_{23}^2).\label{div condition sl2}
\end{align}
\end{example}

In general, the condition for a coset to define a subalgebra may be
described by a finite number of inequalities in the $p$-adic values of
polynomials in the matrix entries. If these polynomials are
\emph{monomials} (as is the case for the Heisenberg Lie ring;
cf.~\eqref{div condition heisenberg}), the computation of the local
zeta function reduces to the computation of the generating function of
a rational polyhedral \emph{cone} (or system of linear homogeneous
diophantine equations).
In general, a resolution of singularities -- a tool from algebraic
geometry -- may be used to remedy the situation. It allows for a
partition of the domain of integration into pieces on which the
integral may be expressed in terms of generating functions of
polyhedral cones. The pieces are indexed by the $\Fp$-points of
certain algebraic varieties defined over~$\Fp$. These kinds of
$p$-adic integrals, called \emph{cone integrals}, were introduced
in~\cite{duSG/00}. A comprehensive introduction to cone integrals may
be found in~\cite[Sections 4 and 5]{duSSegal/00}.

The description of local zeta functions of groups and rings in terms
of cone integrals has far reaching applications for the analysis of
analytic properties of global zeta functions (cf.~Section~\ref{section
global}).

\subsection{Local functional equations}\label{subsection local funeqs}

The zeta functions of the rings we have presented so far as examples
all share a remarkable property: their local factors generically
exhibit a palindromic symmetry on inversion of the prime~$p$. More
precisely, almost all of the Euler factors satisfy a local functional
equation of the form
\begin{equation}\label{funeq}
\zeta_{L,p}(s)|_{p\rightarrow
p^{-1}}=(-1)^ap^{b-cs}\zeta_{L,p}(s),
\end{equation}
where $a,b,c$ are integers which are independent of the prime~$p$.  In
the present section we explain and give an outline of the proof of the
following theorem.

\begin{theorem}\cite[Theorem A]{Voll/06a}\label{theorem voll annals}
Let $L$ be a ring of additive rank $n$. There are smooth {\rm
projective} varieties $V_t$, defined over $\Q$, and rational functions
$W_t(X,Y)$ $\in\Q(X,Y)$, $t\in[m]$, such that, for almost all primes
$p$, the following hold:
\begin{enumerate}
\item \begin{equation}\label{formula local projective}
\zeta_{L,p}(s)=\sum_{t=1}^m b_t(p)W_t(p,p^{-s}),
\end{equation}
where $b_t(p)$ denotes the number of $\Fp$-rational points of
$\ol{V_t}$, the reduction modulo $p$ of $V_t$.

\item \label{second point}Setting $b_t(p^{-1}):=p^{-\dim(V_t)}b_t(p)$ the following
functional equation holds:
\begin{equation}\label{funeq annals}
\zeta_{L,p}(s)|_{p\rightarrow p^{-1}}=(-1)^np^{\binom{n}{2}-ns}\zeta_{L,p}(s).
\end{equation}
\end{enumerate}
\end{theorem}

Note that the advance of Theorem~\ref{theorem voll annals} over
Theorem~\ref{theorem duSG} consists in the assertion~\eqref{second
  point}. The notation `$p\rightarrow p^{-1}$' needs some
justification. If~$b_t(p)$ is a polynomial in $p$, $b_t(p^{-1})$ is
with the rational number obtained by evaluating this polynomial at
$p^{-1}$. This follows from the fact that the varieties
$\overline{V_t}$ are smooth and projective. In general, the above
definition is motivated by properties of the numbers of $\Fp$-rational
points of such varieties, which follow from the Weil conjectures. More
precisely, let $V$ be a smooth projective variety defined over the
finite field $\Fp$. By deep properties of the Hasse-Weil zeta function
associated with~$V$, there are complex numbers $\alpha_{rj}$, $0\leq
r\leq 2\dim(V)$, $1\leq j\leq t_r$ for suitable non-negative integers
$t_r$, such that the number $b_V(p)$ of $\Fp$-rational points of~$V$
can be written as
\begin{equation}\label{frobenius eigenvalues}
b_V(p)=\sum_{r=0}^{2\dim(V)}(-1)^r \sum_{j=1}^{t_r}\alpha_{rj}.
\end{equation}
(Note that the numbers $t_r$ may well be zero; cf.~the examples given in Section~\ref{subsection flags}.)  Furthermore, for each
$r\in[2\dim(V)]_0$ the \ul{multi}sets
$$\left\{\alpha_{rj}|\;j\in [t_{2\dim(V)-r}]\right\} \text{ and
}\left\{\frac{p^{\dim(V)}}{\alpha_{rj}}|\;j\in [t_r]\right\}$$
coincide. Thus,
$$b_V(p^{-1}):=p^{-\dim(V)}b_V(p)=\sum_{r=0}^{2\dim(V)}(-1)^r
\sum_{j=1}^{t_r}\alpha_{rj}^{-1}$$
may be interpreted as the
expression we obtain by inverting the terms $\alpha_{rj}$
in~\eqref{frobenius eigenvalues} (even if they are not, in general,
powers of the prime $p$).

Before we give an outline of the proof of Theorem~\ref{theorem
voll annals}, let us revisit Example~\ref{example ec}.

\begin{example} \label{example ec projective}
Let $E$ denote the elliptic curve defined by the equation
$y^2=x^3-x$. For a prime $p$, denote this time by $b(p)$ the number of
\emph{projective} points of $E$ over $\Fp$, i.e.
$$b(p):=|\{(x:y:z)\in\mathbb{P}^2(\Fp)|\;y^2z=x^3-xz^2\}|.$$ Clearly
$b(p)=c(p)+1$, where $c(p)$ was defined in Example~\ref{example ec}:
we simply add the point $(0:1:0)$ `at infinity'. The results quoted
there imply that 
$$b(p)=
\begin{cases}1+p\text{ if }p\equiv3\mod(4)\text{ and}\\
1-(\pi+\ol{\pi})+p\text{ otherwise,}\end{cases}$$ where
$\pi\ol{\pi}=p$. Note that this last equation implies that
$\pi^{-1}=\ol{\pi}/p$ and $\ol{\pi}^{-1}=\pi/p$, so that
$$
1-(\pi^{-1}+\ol{\pi}^{-1})+p^{-1}=p^{-1}(p-(\ol{\pi}+\pi)+1)=p^{-1}b(p)=b(p)|_{p\rightarrow
p^{-1}},$$ by definition of the latter.
\end{example}

\emph{Outline of proof of Theorem~\ref{theorem voll annals}:} (For
details see~\cite[Sections~2 and~3]{Voll/06a}.) The proof falls into
two parts. The first is of a combinatorial and Coxeter group theoretic
nature. It consists in proving the following general result about
generating functions.

\begin{proposition}\label{proposition IP}
Let $n\in\N$ and, let $(W_I(p^{-s}))_{I\subseteq[n-1]}$ be a family of
functions in $p^{-s}$ with the property that
\begin{equation}\label{equation IP}
 \forall I\subseteq[n-1]:\;W_I(p^{-s})|_{p\rightarrow
 p^{-1}}=(-1)^{|I|}\sum_{J\subseteq I}W_J(p^{-s}).
\end{equation}
Then the function 
\begin{equation}\label{equation igusa funeq}
W(p^{-s}):=\sum_{I\subseteq[n-1]}\binom{n}{I}_{p^{-1}}W_I(p^{-s})
\end{equation} 
(with the polynomials $\binom{n}{I}_{X}$ defined as
in~\eqref{definition flag polynomials}) satisfies
\begin{equation}\label{equation funeq} 
  W(p^{-s})|_{p\rightarrow p^{-1}}=(-1)^{n-1}p^{\binom{n}{2}}W(p^{-s}).
\end{equation}
\end{proposition}

\begin{remark}
We do not need to specify the operation $p\rightarrow p^{-1}$
in~\eqref{equation IP} at this stage; the left hand sides of these
equations could be \emph{defined} in terms of the right hand sides. We
do not assume the functions $W_I(p^{-s})$ to be rational
in~$p^{-s}$. In practice, we will apply Proposition~\ref{proposition
  IP} to families of rational functions $W_I(p^{-s})$ which are
themselves of the form~\eqref{formula local projective}, and we define
$p\rightarrow p^{-1}$ as in Theorem~\ref{theorem voll annals}. What is
understood, however, is that the inversion of the prime extends
linearly to $W(p^{-s})$, and that, of course,
$\binom{n}{I}_{p^{-1}}|_{p\rightarrow p^{-1}}=\binom{n}{I}_{p}$.
\end{remark}

\begin{proof}[Proof of Proposition~\ref{proposition IP}] 
We utilise the Coxeter group theoretic description of the numbers
$\binom{n}{I}_p$ given in Proposition~\ref{proposition binomial
length}. It is a well-known fact~(\cite[Section 1.8]{Humphreys/90})
that there is a unique longest element $w_0\in S_n$, namely the
inversion, such that, for all $w\in S_n$,
\begin{equation}\label{length inversion}
\ell(w)+\ell(ww_0)=\ell(w_0)=\binom{n}{2}
\end{equation} and
\begin{equation}\label{descent type inversion}
D_L(ww_0)=D_L(w)^c.
\end{equation}
Here, given $I\subseteq[n-1]$ we write $I^c$ for $[n-1]\setminus
I$. We also need the following Lemma.
\begin{lemma}\cite[Lemma 7]{VollBLMS/06}\label{lemma BLMS} 
Under the hypotheses of Proposition \ref{proposition IP}, for all
$I\subseteq[n-1]$,
$$\sum_{I\subseteq J}W_J(p^{-s})|_{p\rightarrow
p^{-1}}=(-1)^{n-1}\sum_{I^c\subseteq J}W_J(p^{-s}).$$
\end{lemma}

\begin{proof} We have
\begin{equation*}
\sum_{I\subseteq J}W_J(p^{-s})|_{p\rightarrow p^{-1}}=\sum_{I\subseteq
J}(-1)^{|J|}\sum_{S\subseteq J}W_S(p^{-s})=\sum_{R\subseteq[n-1]}c_R
W_R(p^{-s}),
\end{equation*}
say, where 
\begin{multline*}
c_R=\sum_{R\cup I\subseteq J}(-1)^{|J|}=(-1)^{|R\cup
I|}\sum_{S\subseteq(R\cup I)^c}(-1)^{|S|}\\=(-1)^{|R\cup
I|}(1-1)^{|R\cup I|^c}=\begin{cases} (-1)^{n-1}&\text{ if
}R\supseteq I^c,\\0&\text{ otherwise.}\end{cases}
\end{multline*}
This proves Lemma~\ref{lemma BLMS}.
\end{proof}
We compute
\begin{align*}
W(p^{-s})|_{p\rightarrow
p^{-1}}&=\sum_{I\subseteq[n-1]}\binom{n}{I}_pW_I(p^{-s})|_{p\rightarrow
p^{-1}}&&\eqref{equation igusa funeq}\\
&=\sum_{I\subseteq[n-1]}\left(\sum_{w\in S_n,\,D_L(w)\subseteq
I}p^{\ell(w)}\right)W_I(p^{-s})|_{p\rightarrow
p^{-1}}&&\text{Prop. \ref{proposition binomial length}}\\ &=\sum_{w\in
S_n}p^{\binom{n}{2}-\ell(ww_0)}\sum_{D_L(w)\subseteq
I}W_I(p^{-s})|_{p\rightarrow p^{-1}}&&\eqref{length inversion}\\
&=(-1)^{n-1}p^{\binom{n}{2}}\sum_{w\in
S_n}p^{-\ell(ww_0)}\sum_{D_L(ww_0)\subseteq I}W_I(p^{-s})&&\text{Lemma
\ref{lemma BLMS}}, \eqref{descent type inversion}\\
&=(-1)^{n-1}p^{\binom{n}{2}}\sum_{I\subseteq[n-1]}\left(\sum_{w\in
S_n,\,D_L(ww_0)\subseteq I}p^{-\ell(ww_0)}\right)W_I(p^{-s})\\
&=(-1)^{n-1}p^{\binom{n}{2}}W(p^{-s}).
\end{align*}
This proves Proposition~\ref{proposition IP}.
\end{proof}

We now proceed to the second part of the proof of Theorem~\ref{theorem
voll annals}. It consists in proving that the local zeta function
$\zeta_{L,p}(s)$ of a ring $L$ of additive rank $n$ may be written as
$$(1-p^{-ns})^{-1}W(p^{-s}),$$ where $W(p^{-s})$ is of the
form~\eqref{equation igusa funeq} for suitable (rational) functions
$W_I(p^{-s})$, satisfying the hypotheses~\eqref{equation IP} of
Proposition~\ref{proposition IP}.  This will require both
algebro-geometric and combinatorial methods (which are similar to but
markedly different from the ones used to study cone integrals). It may
be instructive to see this done in a familiar special case first.

\begin{example} \label{example igusa abelian}
  We will see below in Example~\ref{example igusa abelian II} that
  the local zeta functions of the abelian group $\Z^n$ may be written
  as
$$\zeta_{\Z^n,p}(s)=\frac{1}{1-X_n}\sum_{I\subseteq[n-1]}\binom{n}{I}_{p^{-1}}\prod_{\iota\in
I}\frac{X_\iota}{1-X_\iota},$$ where, for $i\in[n]$, $X_i:=p^{i(n-i)-is}$.
One checks immediately that the functions
\begin{equation}\label{WI abelian}
W_I(p^{-s}):=\prod_{\iota\in I}\frac{X_\iota}{1-X_\iota}
\end{equation}
satisfy~\eqref{equation IP}. Indeed, the operation $p\rightarrow
p^{-1}$ simply amounts to an inversion of the `variables' $X_i$, as
they are monomials in~$p$ and $p^{-s}$, and
$$\frac{X^{-1}}{1-X^{-1}}=-\left(1+\frac{X}{1-X}\right).$$ More
conceptually, the validity of the equations~\eqref{equation IP} may be
regarded as a consequence of Theorem~\ref{theorem stanley} in the
special case studied in Example~\ref{example trivial cone}, as we may
view $W_I(p^{-s})$ as obtained from the rational generating function
in variables $X_i$, counting positive integral solutions of an (empty)
set of linear homogeneous diophantine equations in $|I|$ variables,
where the variables $X_i$ are substituted by certain monomials in $p$
and $p^{-s}$. A variation of this basic idea will be crucial for the
proof of Theorem~\ref{theorem voll annals}.
\end{example}

Given a prime $p$, our task is to enumerate full additive sublattices
of the $n$-dimensional $\Zp$-algebra $L_p\subset\Qp\otimes L_p$ which
are subalgebras, i.e.~which are closed under multiplication. It is
easy to see that, given such a lattice $\L$, there is a unique lattice
$\L_0$ in the \emph{homothety class} $[\L]:=\Qp^*\L$ of $\L$ such that
the subalgebras contained in $[\L]$ are exactly the multiples
$p^m\L_0$, $m\in\N_0$. Indeed, given any sublattice $\L$ of $L_p$, and
$e\in\Z$, clearly $(p^e\L)^2\subseteq p^e\L$ if and only if $p^e\L^2
\subseteq \L$. Let $e_0:=\min\{e\in\Z|\;p^{e}\L\subseteq L_p\text{ and
}p^{e}\L^2 \subseteq \L\}$, and set $\L_0:=p^{e_0}\L$. Evidently,
$\L_0$ only depends on the homothety class of~$\L$. We thus have
$$\zeta_{L_p}(s)=(1-p^{-ns})^{-1}\sum_{[\L]}|L_p:\L_0|^{-s}.$$ We set
$$W(p^{-s}):=\sum_{[\L]}|L_p:\L_0|^{-s}.$$ It remains to show that
$W(p^{-s})$ is of the form~\eqref{equation igusa funeq}, with rational
functions $W_I(p^{-s})$ to which Proposition~\ref{proposition IP} is
applicable. We will achieve this by first partitioning the set of
homothety classes of lattices into finitely many parts, indexed by the
subsets $I$ of $[n-1]$, reflecting (aspects of) their elementary
divisor types. On each of these parts, we will describe the indices
$|L_p:\L_0|$ in terms of algebraic congruences, and then encode the
numbers of solutions to these congruences in terms of a suitable
$p$-adic integral $W_I(p^{-s})$ so that the family
$(W_I(p^{-s}))_{I\subseteq[n-1]}$ satisfies the `inversion
properties'~\eqref{equation IP}. The proof of the latter will require
sophisticated methods from algebraic geometry, which we can only
sketch here.

The reader will note the analogy with the proof of
equation~\eqref{equation igusa}, which also proceeded by expressing
the numbers of certain congruences in terms of the Haar measure of
suitable sets.

We recall from Section~\ref{subsection counting} that lattices in
$L_p$ are in $1-1$-correspondence with cosets $\G M$, where
$\G=\GL_n(\Zp)$ and $M\in\trnzp$, where the rows of $M$ encode
coordinates of generators of $\L$ with respect to a fixed basis
$(l_1,\dots,l_n)$ for $L_p$ as $\Zp$-module. For $r\in[n]$, let $C_r$
denote the matrix of the linear map given by right-multiplication with
the basis element~$l_r$ with respect to this basis. It is then not
hard to show (cf.~the proof of \cite[Theorem 5.5]{duSG/00}) that the
lattice corresponding to the coset $\G M$ is a sub\emph{algebra} if
and only if
\begin{equation}\label{subalgebra condition}
\forall i,j\in[n]:\; M_i\sum_{r\in[n]} C_r m_{jr} \in\langle
M_k|\;k\in[n]\rangle_{\Zp},
\end{equation}
where $M_i$ denotes the $i$-th row of $M$. This condition is easy to
check if $M$ may be chosen to be diagonal; in this case,
condition~\eqref{subalgebra condition} is satisfied if, for all
$k\in[n]$, the $k$-th entries of all the vectors on the left hand side
are divisible by $M_{kk}$, the $k$-th diagonal entry of $M$. In
general, however, the coset $\G M$ will not contain a diagonal
element. One way around this is to choose a different basis
for~$L_p$. Indeed, by the Elementary Divisor Theorem, the coset $\G M$
does contain an element of the form $D\alpha^{-1}$, where
$\alpha\in\G$ and
$$
D=D(I,\bfr_0)=p^{r_0}\,\text{diag}(\underbrace{\underbrace{p^{\sum_{\iota\in
I}r_\iota},\dots,p^{\sum_{\iota\in
I}r_\iota}}_{i_1},\dots,p^{r_{i_l}},\dots,p^{r_{i_l}}}_{i_l},1,\dots,1)$$
for a set $I=\{i_1,\dots,i_l\}_<\subseteq[n-1]$ and a vector
$(r_0,r_{i_1},\dots,r_{i_l})=:\bfr_0\in\N_0\times\N^l$ (both depending
only on $\G M$). Setting $\bfr:=(r_{i_1},\dots,r_{i_l})$, we say that
the homothety class $[\Lambda]$ of $\Lambda$ is of type $(I,\bfr)$ (or
sometimes, by abuse of notation, of type $I$) and write
$\nu([\Lambda])=(I,\bfr)$ (or $\nu([\Lambda])=I$, respectively). The
matrix $\alpha$ is only unique up to right-multiplication by an
element of
$$\G_{I,\bfr}:=\left\{\left(\begin{tabular}{c|c|c|c|c}
$\gamma_{i_1}$&$*$&$\cdots$&$*$&$*$\\\hline
$ p^{r_{i_1}}*$ &$\gamma_{i_2-i_1}$&$\ddots$&$\vdots$&\\\hline
$p^{r_{i_1}+r_{i_2}}*$&$p^{r_{i_2}}*$&$\ddots$&$*$&$\vdots$\\\hline
$\vdots$&$\vdots$&$\ddots$&$\gamma_{i_l-i_{l-1}}$&$*$\\\hline
$p^{r_{i_1}+\dots+r_{i_l}}*$&$p^{r_{i_2}+\dots+r_{i_l}}*$&$\cdots$&$p^{r_{i_l}}*$&$\gamma_{n-i_l}$
\end{tabular}\right)\right\},$$
where $\gamma_i\in\G_i:=\GL_i(\Zp)$, and $*$ stands for an arbitrary
matrix with entries in~$\Zp$ of the respective size. As an immediate
and useful corollary, we deduce a formula for the number of lattices
of given type $(I,\bfr)$:
\begin{equation}\label{equation index}
\left|\left\{[\L]|\;\nu([\L])=(I,\bfr)\right\}\right|=|\G:\G_{I,\bfr}|=\mu(\G)/\mu(\G_{I,\bfr})=\binom{n}{I}_{p^{-1}}p^{\sum_{\iota\in
    I}r_\iota \iota(n-\iota)}.
\end{equation} 
Here, $\mu$ denotes the Haar measure on the group $\G$ normalised so
that $\mu(\G)=(1-p^{-1})\cdots(1-p^{-n})$. It is a crucial observation
that it coincides with the restriction of the additive Haar measure on
$\text{Mat}_n(\Zp)\cong\Zp^{n^2}$, normalised so that
$\mu(\text{Mat}_n(\Zp))=1$.

We consider the $n\times n$-matrix of $\Z$-linear forms
$$\mcR(\bfy)=(\tuL_{ij}(\bfy))\in\text{Mat}_n(\Z[\bfy]),$$ where
$\tuL_{ij}(\bfy):=\sum_{k\in[n]}\lambda_{ij}^ky_k$, encoding the
structure constants $\lambda_{ij}^k$ of $L$ with respect to the chosen
basis, that is $l_il_j=\sum_{k\in[n]}\lambda_{ij}^kl_k$. 
Right-multiplication by $\alpha$ now yields that the subalgebra
condition~\eqref{subalgebra condition} is equivalent to
\begin{equation}\label{subalgebra
conditionII} \forall i\in[n]:\; D \mcR_{(i)}(\alpha) D \equiv 0 \mod
(D_{ii}),
\end{equation}
where $\mcR_{(i)}(\alpha):=\alpha^{-1}\mcR(\alpha[i])(\alpha^{-1})^t$,
as a quick calculation shows. (Here we write $\alpha[i]$ for the
$i$-th column of the matrix~$\alpha$.) Considering these matrix
congruences modulo a common modulus, this is equivalent to
\begin{multline}
\forall
  i,r,s\in[n]:\\(\mcR_{(i)}(\alpha))_{rs}\,p^{r_0+\sum_{s\leq\iota\in
  I}r_\iota + \sum_{r\leq\iota\in I}r_\iota + \sum_{i>\iota\in
  I}r_\iota}\equiv 0 \mod (p^{\sum_{\iota\in
  I}r_\iota})\label{subalgebra conditionIII}
\end{multline}
which may in turn be
  reformulated as
\begin{multline}\label{subalgebra conditionIV}
r_0\geq\\\sum_{\iota\in I}r_\iota-\underbrace{\min\left\{\sum_{\iota\in
I}r_\iota,\sum_{s\leq\iota\in I}r_\iota + \sum_{r\leq\iota\in I}r_\iota +
\sum_{i>\iota\in
I}r_\iota+v_{irs}(\alpha)|\,(i,r,s)\in[n]^3\right\}}_{=:m([\L])},
\end{multline}
where
$v_{irs}(\alpha):=\min\left\{v_p\left((\mcR_{(\iota)}(\alpha))_{\rho
\sigma}\right)|\iota\leq i,\rho\geq r,\sigma\geq s\right\}$.

We observe that this description of the quantity $m([\Lambda])$ is in
terms which are linear in the $(r_\iota)_{\iota\in I}$ and terms
$v_{irs}(\alpha)$, which only depend on $\alpha$. Moreover, by
construction the $v_{irs}(\alpha)$ only depend on the coset $\alpha
B$, where $B\subset\GL_n(\Zp)$ is the Borel subgroup of
upper-triangular matrices. (This is the purpose of using inequalities
rather than equalities in their definition.) As in the proof of the
identity~\eqref{equation igusa}, this allows us to express the numbers
of lattice classes $[\Lambda]$ of given type $\nu([\Lambda])$ and
invariant $m([\Lambda])$ in terms of the Haar measure of the set on
which the integrand of a certain $p$-adic integral is constant. More
precisely, we set, for $(i,r,s)\in[n]^3$,
$$\bff_{irs}(\bfy):=\{(\mcR_{(\iota)}(\bfy))_{\rho\sigma}|\;\iota\leq
  i,\rho\geq r,\sigma\geq s\}$$ and
\begin{multline}\label{ZI formula}Z_I((s_\iota)_{\iota\in I},s_n):=\\\int_{p\Zp^{l}\times\G}\prod_{\iota\in I}|x_\iota|_p^{s_\iota}\left\|\left\{\prod_{\iota\in I}x_i\right\}\cup \bigcup_{(i,r,s)}\left(\prod_{\iota\in
  I}x_\iota^{\delta_{\iota\geq r}+\delta_{\iota\geq
  s}+\delta_{\iota<i}}\right)\bff_{irs}(\bfy)\right\|_p^{s_n}d\bfx_I
  d\bfy.
\end{multline}

(Here, we extended the $p$-adic absolute value to a set $\mcS$ of
$p$-adic numbers by setting $\|S\|_p:=\min\{|s|_p|\;s\in\mcS\}$. We
denoted by $d\bfx_I=dx_{i_1}\cdots dx_{i_l}$ the Haar measure on
$p\Zp^l$.) This $p$-adic integral has been expressly set up so that,
for each $I\subseteq[n-1]$,
\begin{align*}
\sum_{\nu([\L])=I}|L_p:\L_0|^{-s}=\binom{n}{I}_{p^{-1}}W_I(p^{-s}),
\end{align*}
say, where
$$W_I(p^{-s}):=\frac{Z_I((s(\iota+n)-\iota(n-\iota)-1)_{\iota\in
I},-sn)}{(1-p^{-1})^{l}\mu(\G)},$$ so that
$$W(p^{-s})=\sum_{I\subseteq[n-1]}\binom{n}{I}_{p^{-1}}W_I(p^{-s}).$$
We would like to establish that the functions $W_I(p^{-s})$ satisfy the
inversion property~\eqref{equation IP}. Let us first confirm this in
the abelian case.

\begin{example}[abelian groups revisited] \label{example igusa abelian II} 
If $L_p$ is abelian, i.e.~if the multiplication on $L_p$ is trivial,
all the sets of polynomials $\bff_{irs}$ are equal to $\{0\}$,
so~\eqref{ZI formula} takes the form
\begin{multline*}
Z_I((s_\iota)_{\iota\in I},s)=\int_{(p\Zp)^l\times\G}\prod_{\iota\in
I}|x_\iota|_p^{s_\iota+s_n}d\bfx_I d\bfy\\=\mu(\G)\prod_{\iota\in
I}\int_{p\Zp}|x_\iota|_p^{s_\iota+s_n}dx_\iota=\mu(\G)(1-p^{-1})^l\prod_{\iota\in
I}\frac{p^{-1-s_\iota-s_n}}{1-p^{-1-s_\iota-s_n}}
\end{multline*}
and thus (cf.~\eqref{WI abelian})
\begin{equation}\label{WI abelian II}
W_I(p^{-s})=\prod_{\iota\in
I}\frac{p^{\iota(n-\iota)-s\iota}}{1-p^{\iota(n-\iota)-s\iota}}.
\end{equation}
(Of course we could have deduced this immediately from~\eqref{equation
index}, avoiding any reference to $p$-adic integrals.)
\end{example}

We note that in the formula~\eqref{ZI formula}, the variables $\bfx$
enter \emph{monomially}. If the same were true for the variables
$\bfy$, the inversion properties~\eqref{equation IP} would follow from
the following proposition, generalising a result of Stanley:

\begin{proposition}\cite[Proposition~2.1]{Voll/06a}\label{proposition stanley}
  Let $s,t\in\N_0$ and, for $\sigma\in[s]$, $\tau\in[t]$, let
  $L_{\sigma\tau}(\bfn)$ be $\Z$-linear forms in the variables
  $n_1,\dots,n_r$. Let $X_1,\dots,X_r$, $Y_1,\dots,Y_s$ be independent
  variables and set
\begin{align*}
  Z^\circ(\bfX,\bfY):=&\sum_{\bfn\in\N^r}\prod_{\rho\in[r]}
  X_\rho^{n_\rho}\prod_{\sigma\in[s]}Y_\sigma^{\min_{\tau\in[t]}\{L_{\sigma\tau}(\bfn)\}},\\
  Z(\bfX,\bfY):=&\sum_{\bfn\in\N_0^r}\prod_{\rho\in[r]}
  X_\rho^{n_\rho}\prod_{\sigma\in[s]}Y_\sigma^{\min_{\tau\in[t]}\{L_{\sigma\tau}(\bfn)\}}.
\end{align*}
Then
$$Z^\circ(\bfX^{-1},\bfY^{-1}) = (-1)^r Z(\bfX,\bfY).$$
\end{proposition}
For $t\leq1$ this follows immediately from Theorem~\ref{theorem
stanley}, as $Z^\circ(\bfX,\bfY)$ (and $Z(\bfX,\bfY)$) may be
interpreted in terms of the generating functions $\ol{E}(\bfX)$ (and
$E(\bfX)$, respectively) associated with the empty set of equations in
$r$ variables. The general case follows from an adaptation of the
proof of~\cite[Proposition 4.16.14]{Stanley/97}.

In general, the inversion properties~\eqref{equation IP} can be proved
by making the integral~\eqref{ZI formula} `locally monomial' in the
variables~\bfy. This is achieved by applying a `principalisation of
ideals', a tool from algebraic geometry. More precisely, we apply the
following deep result to the ideal
$\prod_{(i,r,s)\in[n]^3}(\bff_{irs}(\bfy))$, defining a subvariety of
the homogeneous space $X=\GL_n/B$.

\begin{theorem}\cite[Theorem 1.0.1]{Wlodarczyk/05}\label{theorem principalisation}
  Let $\mcI$ be a sheaf of ideals on a smooth algebraic
  variety~$X$. There exists a principalisation $(Y,h)$ of $\mcI$, that
  is, a sequence
  $$
  X=X_0\stackrel{h_1}{\longleftarrow}X_1\longleftarrow\dots\stackrel{h_\iota}{\longleftarrow}X_{\iota}\longleftarrow\dots\stackrel{h_r}\longleftarrow
  X_r=Y$$ of blow-ups $h_{\iota}:X_{{\iota}}\rightarrow X_{\iota-1}$
  of smooth centres $C_{{\iota}-1}\subset X_{{\iota}-1}$ such that
\begin{itemize}
\item[a)] The exceptional divisor $E_{\iota}$ of the induced morphism
  $h^{\iota}=h_{\iota}\circ\dots\circ h_{1}:X_{\iota}\rightarrow X$
  has only simple normal crossings and $C_{\iota}$ has simple normal
  crossings with~$E_{\iota}$.
\item[b)] Setting $h:=h_r\circ\dots\circ h_1$, the total transform
  $h^{*}(\mcI)$ is the ideal of a simple normal crossing
  divisor~$\wt{E}$. If the subscheme determined by $\mcI$ has no
  components of codimension one, then $\wt{E}$ is an $\N$-linear
  combination of the irreducible components of the divisor~$E_r$.
\end{itemize}
\end{theorem}

The existence of a principalisation lies as deep as Hironaka's
celebrated resolution of singularities in characteristic
zero~\cite{Hironaka/64}. See~\cite{Wlodarczyk/05} for details.

\subsection{A class of examples: $3$-dimensional $p$-adic Lie algebras}\label{section 3D}
Constructing an explicit principalisation for a given family of ideals
$(\bff_{irs}(\bfy))_{irs})$ is in general very difficult. In the
special case that $L_p$ is an anti-symmetric (not necessarily
nilpotent or Lie) $\Zp$-algebra of dimension~$3$, however, the
approach of Theorem~\ref{theorem voll annals} leads to an explicit,
unified expression for the zeta function
of~$L_p$. 

\begin{theorem} \label{3D theorem}\cite[Theorem 1]{KlopschVoll/07}
  Let $L$ be a $3$-dimensional $\Zp$-Lie algebra. Then there is a
  ternary quadratic form $f(\bfx)\in\Zp[x_1,x_2,x_3]$, unique up to
  equivalence, such that, for $i\geq 0$,
  \begin{equation*}
   \zeta_{p^iL}(s)=\zeta_{\Zp^3}(s)-Z_f(s-2)\zeta_p(2s-2)\zeta_p(s-2)p^{(2-s)(i+1)}(1-p^{-1})^{-1},
  \end{equation*}
 where $Z_f(s)$ is Igusa's local zeta function associated with $f$.
\end{theorem}
The form $f(\bfx)$ in Theorem~\ref{3D theorem} may be defined
explicitly in terms of the structure constants of $L_p$ with respect
to a chosen basis; different bases give rise to equivalent forms
(see~\cite{KlopschVoll/07} for details).

This result yields, in particular, a uniform expression for the zeta
functions of all $3$-di\-men\-sio\-nal $\Zp$-(Lie) algebras we have
seen so far (and others, e.g.~\cite{Klopsch/03}). For example, the
forms $f(\bfx)$ for the abelian algebra $\Zp^3$, the Heisenberg Lie
algebra and the `simple' Lie algebra $\sltwo(\Zp)$ are $0$, $x_3^2$ and
$x_3^2-4x_1x_2$, respectively.

Using the setup of Section~\ref{subsection local funeqs}, the key to
proving Theorem~\ref{3D theorem} is the observation that only the
functions $W_I(p^{-s})$ with $1\in I$ differ from the `abelian'
functions~\eqref{WI abelian}. Indeed, if $1\not\in I$, the
conditions~\eqref{subalgebra conditionIII} hold for all
$r_0\in\N_0$. If $1\in I$ then they hold if and only if
\begin{equation}\label{3D subalgebra}
  p^{r_0}(\mcR_{(1)}(\alpha))_{23}\equiv0\mod (p^{r_1})
\end{equation}
and a quick calculation shows that, for $\alpha=(\alpha_{ij})\in\G_3$,
$$
\det(\alpha)(\mcR_{(1)}(\alpha))_{23}=\tuL_{23}(\alpha[1])\alpha_{11}-\tuL_{13}(\alpha[1])\alpha_{21}+\tuL_{12}(\alpha[1])\alpha_{31}.$$
Setting
\begin{equation*}
f(\bfx):=\tuL_{23}(\bfx)x_1-\tuL_{13}(\bfx)x_2+\tuL_{12}(\bfx)x_3
\end{equation*}
we see that~\eqref{3D subalgebra} holds if and only if 
$$r_0\geq r_1-v_p(f(\alpha[1])).$$ The computation of the
integral~\eqref{ZI formula} is thus no harder than the computation of
the Igusa zeta function associated with the quadratic polynomial
$f(\bfx)$.

We note that Theorem~\ref{3D theorem} also yields a complete
description of the possible poles of zeta functions of $3$-dimensional
$\Zp$-Lie algebras, as the poles of Igusa's local zeta function of
quadratic forms are well understood
(cf.~\cite[Corollary~1.2]{KlopschVoll/07}). In higher dimensions, such
a description is entirely elusive.  Also, Theorem~\ref{3D theorem}
shows explicitly the relationship between $\zeta_L(s)$ and
$\zeta_{pL}(s)$ if $L$ is of dimension~$3$. No such formula is known
in higher dimensions.

\subsection{Global zeta functions of groups and rings}\label{section global}

Let $G$ be a group with polynomial subgroup growth. As noted in the
introduction, the degree of polynomial subgroup growth of~$G$ is
encoded in an analytic invariant of the group's zeta function, namely
its abscissa of convergence~$\alpha$ (cf.~\eqref{alpha
definition}). The following is a deep result.

\begin{theorem}\cite[Theorem 1.1]{duSG/00} \label{theorem duSG analytic} 
Let $G$ be a $\T$-group.
\begin{enumerate}
\item The abscissa
  of convergence $\alpha$ of its zeta function $\zeta_G(s)$ is a
  rational number, and $\zeta_G(s)$ can be meromorphically continued
  to $\mathfrak{Re}(s)>\alpha-\delta$ for some $\delta>0$. The
  continued function is holomorphic on the line
  $\mathfrak{Re}(s)=\alpha$ except for a pole at $s=\alpha$.
\item Let $b+1$ denote the multiplicity of the pole of $\zeta_G(s)$ at
$s=\alpha$. There exists a real number $c\in\mathbb{R}$ such that
$$\sum_{i\leq m}a_i \sim c \cdot m^\alpha (\log m)^b\quad\text{as $m\rightarrow\infty$}.$$ 
\end{enumerate}
\end{theorem}
The proof of Theorem~\ref{theorem duSG analytic} given
in~\cite{duSG/00} proceeds via an analysis of the (local) `cone
integrals' mentioned above.

Whilst it is a remarkable fact that global zeta functions of nilpotent
groups always allow for some analytic continuation beyond their
abscissa of convergence, it is not the case that they may all be
continued to the whole complex plane, as is the case for abelian
groups or the Heisenberg group. In fact, numerous groups have been
found for which there are natural boundaries for analytic continuation
(cf.~\cite[Chapter 7]{duSautoyWoodward/08}). Surprisingly little is
known about the abscissa of convergence $\alpha$ and the pole order
$b+1$ in general.

\section{Variations on a theme}\label{section variations}
The theme of counting subobjects of finite index in a nilpotent group
or a ring may be varied in several interesting ways. 

\subsection{Normal subgroups and ideals}\label{subsection normal subgroups} 
One of the forerunners of the very concept of the zeta function of a
group is the Dedekind zeta function of a number field, one of the most
classical objects in algebraic number theory. Given a number field~$k$, with ring of integers $\mathcal{O}$, the Dedekind zeta function
of $k$ is defined as the Dirichlet series
 $$\zeta_k(s):=\sum_{\mathfrak{a}\nl_f\mathcal{O}}
 |\mathcal{O}:\mathfrak{a}|^{-s},$$ where the sum ranges over the
 ideals of finite index in $\mathcal{O}$.  Owing to our understanding
 of the ideal structure in the Dedekind ring $\mathcal{O}$ we have a
 very good control of arithmetic and analytic properties of this
 important function. In particular, we know that it allows for an
 analytic continuation to the whole complex plane and has a simple
 pole at $s=1$. Its residue at this pole encodes important arithmetic
 information about the number field~$k$, given by the class number
 formula.

 Given a general ring $L$, its ideal zeta function is defined as the
 Dirichlet series
 $$\zeta^\nl_{L}(s):=\sum_{m=1}^\infty
 b^\nl_mm^{-s}=\sum_{H\triangleleft_f L}|L:H|^{-s},$$ where
 $b^\nl_m=b^\nl_m(L)$ is the number of ideals in $L$ of index~$m$.
 Similarly, the normal zeta function of a nilpotent group $G$ is
 defined as
 $$\zeta^\nl_{G}(s):=\sum_{m=1}^\infty
a^\nl_mm^{-s}=\sum_{H\triangleleft_f G}|G:H|^{-s},$$ where
$a^\nl_m=a^\nl_m(G)$ denotes the number of normal subgroups of $G$ of
index~$m$.

Both the ideal zeta function of a ring $L$ and the normal zeta
function of a nilpotent group $G$ satisfy an Euler product
decomposition
 $$\zeta^\nl_{L}(s)=\prod_{p \text{ prime}}\zeta^\nl_{L,p}(s),\quad\quad\quad
 \zeta^\nl_{G}(s)=\prod_{p \text{ prime}}\zeta^\nl_{G,p}(s)$$ into
 local factors enumerating subobjects of $p$-power index. Fortunately,
 also the study of normal subgroup growth can be linearised using the
 Lie ring introduced in Section~\ref{section
 linearisation}. By~\cite[Section 4]{GSS/88} we have, for almost all
 primes $p$,
 \begin{equation}\label{equation malcev normal subgroups}
  \zeta^\nl_{G,p}(s)=\zeta^\nl_{L(G),p}(s)
 \end{equation}
where $L(G)$ is the nilpotent Lie ring associated with $G$
(cf.~Section~\ref{section linearisation}).

 \begin{example}\label{example heisenberg normal}
 Let $G$ be the discrete Heisenberg group from Example \ref{example
   heisenberg}. It can be shown that
 $$\zeta^\nl_{G}(s)=\zeta^\nl_L(s)=\zeta(s)\zeta(s-1)\zeta(3s-2).$$
 \end{example}
 \noindent Note again that the equation~\eqref{equation malcev normal
 subgroups} holds for all primes~$p$.

In many ways, the theory of ideal zeta functions of nilpotent groups
is similar to the theory of their (subgroup) zeta functions. In
particular, their local factors are also rational in $p^{-s}$, and
analogues of Theorems~\ref{theorem duS denominators}, \ref{theorem
  duSG} and \ref{theorem duSG analytic} hold. The first explicitly
computed example of a non-uniform zeta function is the normal zeta
function of a class-$2$-nilpotent group.

\begin{example}\label{example nonuniform}
In ~\cite{duS-ecI/01} du Sautoy showed that both the subgroup and the
normal subgroup zeta function of the following class-$2$-nilpotent
group are not finitely uniform. He defined
\begin{equation*}
G := \langle x_1,\dots,x_6,y_1,y_2,y_3|\;\forall i,j: [x_i,x_j]=\mcR(\bfy)_{ij},\;\text{all other $[\,,]$ trivial}\rangle,
\end{equation*}
where
$$\mcR(\bfy)=\left(\begin{array}{cc}0&R(\bfy)\\-R(\bfy)^t&0\end{array}\right)
\quad\text{ with }\quad
R(\bfy)=\left(\begin{array}{ccc}y_3&y_1&y_2\\y_1&y_3&0\\y_2&0&y_1\end{array}\right).$$
Notice that the polynomial $\det(R(\bfy))=y_1y_3^2-y_1^3-y_2^2y_3$
defines the projective elliptic curve $E$ considered in
Example~\ref{example ec projective}. It can be shown
(cf.~\cite[p.~1031]{Voll/04}) that, for $p\not=2$,
$$\zeta^\tl_{G,p}(s)=\zeta_{\Zp^6}(s)(W_1(p,p^{-s}) + b(p) W_2(p,p^{-s})),$$
where $b(p)$ is the number defined in Example~\ref{example ec projective} and
\begin{align*}
W_1(X,Y)&=\frac{1+X^6Y^7+X^7Y^7+X^{12}Y^8+X^{13}Y^8+X^{19}Y^{15}}{(1-X^{18}Y^9)(1-X^{14}Y^8)(1-X^8Y^7)}\\
W_2(X,Y)&=\frac{(1-Y^2)X^6Y^5(1+X^{13}Y^8)}{(1-X^{18}Y^9)(1-X^{14}Y^8)(1-X^8Y^7)(1-X^7Y^5)}.
\end{align*}
Using the identity $b(p)|_{p\rightarrow p^{-1}}=p^{-1}b(p)$
established earlier, the functional equation
\begin{equation}\label{normal funeq}
\zeta^\triangleleft_{G,p}(s)|_{p\rightarrow
p^{-1}}=-p^{36-15s}\zeta^\triangleleft_{G,p}(s)
\end{equation} 
follows immediately. The local subgroup zeta functions
$\zeta_{G,p}(s)$ have not been calculated explicitly.
\end{example}

The methods used to perform the calculations in Example~\ref{example
nonuniform} rely on the fact that the (square root of the) determinant
of the matrix of relations $\mcR(\bfy)$ defines a smooth hypersurface
in the projective space over the centre of the group.  Together with
the algebro-geometric fact that every smooth plane curve defined over
$\Q$ may be defined by the determinant of a suitable matrix of linear
forms, one can, in this way, force any such curve to take on the role
played by the elliptic curve in Example~\ref{example nonuniform} in
the normal zeta function of a class-$2$-nilpotent group. We refer
to~\cite{Voll/05} for details.

Equation~\eqref{normal funeq} is a special case of an analogue of
Theorem~\ref{theorem voll annals} for normal zeta functions of
class-$2$-nilpotent Lie rings (\cite[Theorem~C]{Voll/06a}). To what
extent this symmetry phenomenon extends to normal zeta functions of
other (Lie) rings is largely mysterious. Examples due to
Woodward~(\cite{duSautoyWoodward/08}) show that this may or may not
hold in Lie rings of higher nilpotency classes, and in certain soluble
Lie rings.

\subsection{Representations} \label{subsection reps}
\bigskip Another variant of the theme of counting subgroups in a group
consists in enumerating the group's finite-dimensional irreducible
complex representations. Again, the concept of a zeta function is
helpful to study these if the group has -- at least up to some natural
equivalence relation -- only finitely many irreducible complex
representations of each finite dimension, and if these numbers grow at
most polynomially. We call an (abstract or profinite) group $G$
\emph{rigid} if, for every $n\in\N$, the number $r_n(G)$ of
(continuous, if $G$ is profinite,) irreducible complex representations
of $G$ of dimension $n$ is finite. We say that a rigid group $G$ has
\emph{polynomial representation growth (PRG)} if, for each $m\in\N$,
the number of representations of $G$ of dimension at most $m$ is
bounded above by a polynomial in~$m$, i.e.~ $\sum_{i\leq
  m}r_i(G)=O(1+m^\alpha)$ for some $\alpha\in\mathbb{R}$. As in the
case of counting subgroups, we define a Dirichlet generating function
$$\zeta^{{\rm irr}}_G(s):=\sum_{n=1}^\infty
r_n(G)n^{-s}=\sum_{\rho}(\dim(\rho))^{-s},$$ where $\rho$ ranges over
the finite-dimensional irreducible complex representations of $G$,
called the representation zeta function of $G$. It defines a
convergent function on the complex half-plane determined by the
infimum of these~$\alpha$.

No general characterisation of rigid or PRG groups is known. In the
current section we will concentrate on results regarding three classes
of groups: finitely generated torsion-free nilpotent (or $\T$-)groups,
arithmetic groups and compact $p$-adic analytic groups.

As we will see in Section~\ref{subsubsection reps T-groups},
$\T$-groups are `rigid up to twisting with one-d\-imen\-sio\-nal
representations'. The growth of the numbers of the ensuing equivalence
classes, called `twist-isoclasses', is polynomial, and the associated
representation zeta functions satisfy Euler product decompositions,
indexed by the primes, analogous to the context of counting
subgroups. The Kirillov orbit method offers a suitable `linearisation'
of the problem of counting twist-isoclasses of representations of
$p$-power dimension, and we may once again use our arsenal of tools
from $p$-adic integration to study the Euler factors (at least for
almost all primes).

It is known that arithmetic groups are PRG if and only if they satisfy
the Congruence Subgroup Property (CSP). In Section~\ref{subsubsection
  reps p-adic analytic groups} we review results that show that the
representation zeta functions of these groups, too, satisfy an Euler
product decomposition, indexed by all places of the underlying number
field (including the archimedean ones). The non-archimedean factors
are zeta functions associated with compact $p$-adic analytic
groups. As we shall see, these are also rational functions, albeit not
solely in the parameter $p^{-s}$.

\subsubsection{$\T$-groups} \label{subsubsection reps T-groups}
A $\T$-group has infinitely many one-dimensional irreducible
representations: it has infinite abelianisation, and the group of
one-di\-men\-sio\-nal representations of $\Z^n$, i.e.~of homomorphisms of
$\Z^n$ to $\C^*$, is isomorphic to $(\C^*)^n$. Tensoring with
one-dimensional representations will thus give us an infinitude of
$m$-dimensional representations for every $m$ for which such
representations exist. Fortunately, this is all that needs
fixing. More precisely, given a $\T$-group~$G$, we denote by $R_n(G)$
the set of $n$-dimensional irreducible complex representations
of~$G$. Given $\sigma_1,\sigma_2\in R_n(G)$, we say that $\sigma_1$
and $\sigma_2$ are {\sl twist--equivalent} if there exists a
one-dimensional representation~$\chi\in R_1(G)$ such
that~$\sigma_1=\chi\otimes\sigma_2$. The classes of this equivalence
relation are called twist-isoclasses. The set $R_n(G)$ has the
structure of a quasi--affine complex algebraic variety whose geometry
was analysed by Lubotzky and Magid. They proved in \cite[Theorem
6.6]{LubotzkyMagid/85} that, for every~$m\in\N$, there is a finite
quotient $G(m)$ of $G$ such that every $m$-dimensional irreducible
representation of~$G$ is twist-equivalent to one that factors
through~$G(m)$. In particular, the number $c_m=c_m(G)$ of
twist-isoclasses of irreducible $m$-dimensional representations is
finite. The representation zeta function of $G$ is defined
(cf.~\cite{HrushovskiMartin/07}) by
\begin{equation}\label{definition nilpotent rep zeta function}
\zirr_{G}(s):=\sum_{m=1}^\infty c_mm^{-s}.
\end{equation} Furthermore, the
function~$m\mapsto c_m$ is multiplicative. Indeed, this follows from
Lubotzky and Magid's result together with the group-theoretic fact
that the finite nilpotent groups $G(m)$ are the direct products of
their Sylow $p$-subgroups and the representation-theoretic fact
(\cite[(10.33)]{CurtisReiner-methods/81}) that the irreducible
representations of direct products of finite groups are exactly the
tensor products of irreducible representations of their factors. Thus
$$\zirr_{G}(s)=\prod_{p \text{ prime}}\zirr_{G,p}(s), \quad \text{
  where }\quad\zirr_{G,p}(s):=\sum_{i=0}^\infty c_{p^i}p^{-is}.$$

As in the case of saturable pro-$p$-groups (see Section~10.2 in
Klopsch's lecture notes), there is a close connection between
representations of $\T$-groups and co-adjoint orbits. This
generalisation of Kirillov's orbit method to the discrete setting of
$\T$-groups is due to Howe. In \cite{Howe-nilpotent/77} he shows that
(twist-isoclasses of) irreducible representations in a $\T$-group $G$
are parametrised by co-adjoint orbits of certain (additive) characters
on the associated Lie ring $L(G)$. More precisely, we write
$\widehat{L}$ for the group $\text{Hom}(L,\C^*)$, and $\text{Ad}^*$
for the co-adjoint action of $G$ on $\widehat{L}$. Denote by $L'$ the
Lie subring of $L$ corresponding to the group's derived group~$G'$.
We say that a character $\psi\in\widehat{L}$ is \emph{rational} on
$L'$ if its restriction to $L'$ is a torsion element, i.e.~if
$\psi(nL')\equiv1$ for some $n\in\N$. The smallest such $n$ is called
the \emph{period} of~$\psi$.  Howe's principal result now states that
a character's co-adjoint orbit is finite if and only if the character
is rational on $L'$, and that finite $\text{Ad}^*$-orbits in
$\widehat{L}$, $\Omega$ say, of characters of odd period, are in
$1-1$-correspondence with (twist-isoclasses of) finite-dimensional
representations $U_\Omega$ of $G$ of dimension $|\Omega|^{1/2}$
(see~\cite[Section~3.4]{Voll/06a} for details).

To effectively enumerate twist-isoclasses of finite-dimensional
representations of $G$ we thus have to deal with two problems: given a
character $\psi\in\widehat{L}$ of finite period, we firstly need to
determine the size of its co-adjoint orbit. Secondly, to control
over-counting, we have to determine the size of the co-adjoint orbit of
the restriction of~$\psi$ to $L'$. From now on, we will restrict
ourselves to the case that the nilpotency class of $G$ is~$2$. In this
case, the latter task is trivial, as the co-adjoint action on the
restriction of characters to $L'$, which is central, is trivial.

As in the case of saturable pro-$p$ groups, we associate with a
character $\psi\in\widehat{L}$ the bi-additive antisymmetric map
$$b_\psi: L\times L \rightarrow \C^*, \quad(x,y)\mapsto\psi([x,y]).$$
Note that $b_\psi$ only depends on the restriction of $\psi$
to~$L'$. We define
$$\text{Rad}_\psi:=\text{Rad}(b_\psi)=\{x\in L|\;\forall y\in
L:b_\psi(x,y)=1\}.$$ One can show that, if $\psi$ is rational on $L'$
(so its co-adjoint orbit is finite by Howe's result) and
$|L:\text{Rad}|$ is coprime to finitely many `bad primes', depending
only on $G$, then $\text{Rad}_\psi$ is the Lie ring corresponding to
the stabiliser subgroup $\text{Stab}_G(\psi)$ of $\psi$ under the
co-adjoint action. Then, by the Orbit Stabiliser Theorem, the index
$|L:\text{Rad}_\psi|$ equals the size of the co-adjoint orbit
of~$\psi$. The Kirillov correspondence now implies that the
representation associated with the orbit of $\psi$ has degree
$|L:\text{Rad}_\psi|^{-1/2}$.

Recall that, for a class-$2$-nilpotent group, finite co-adjoint orbits
are pa\-ra\-me\-tri\-zed by rational characters on $L'$ of finite
period. For a prime $p$ and $N\in\N_0$, we write $\Psi_N$ for the set
of $\psi\in\widehat{L'}$ of period~$p^N$. By Howe's results we have

\begin{theorem}\cite[Corollary 3.1]{Voll/06a}\label{repgrowth class 2 nilpotent} Let $G$ be a class-$2$-nilpotent $\T$-group. Then, for almost all
primes~$p$,
\begin{equation}\label{formula reps class 2}\zeta^{\rm{irr}}_{G,p}(s)=\sum_{N\in\N_0,\;\psi\in\Psi_N}|L:\rm{Rad}_\psi|^{-s/2}.
\end{equation}
\end{theorem}

Assume that ${\rm rk}(G/G')=d$ and ${\rm rk}(G')=d'$, say, and let $p$
be a prime for which~\eqref{formula reps class 2} holds. To compute
the right hand side of this equation effectively, we identify $\Psi_N$
with $W_{p,N}:=(\Z/(p^N))^{d'}\setminus p(\Z/(p^N))^{d'}$ as additive
groups, and let $\mcR(\bfy)\in{\rm Mat}(d,\Z[y_1,\dots,y_{d'}])$ be
the matrix of linear forms encoding the commutator structure of $G$,
i.e.~$\mcR(\bfy)_{ij}=\sum_{k=1}^{d'}\lambda_{ij}^ky_k$ if $G$ is
generated by $e_1,\dots,e_d$ subject to the relations
$[e_i,e_j]=\sum_{k=1}^{d'}\lambda_{ij}^kf_k$, say, where $G/G'=\langle
e_1G',\dots,e_dG'\rangle$ and $G'=\langle f_1,\dots,f_{d'}\rangle$.

A simple computation shows that if $\psi\in\Psi_N$ corresponds to
$\bfl\in W_{p,N}$, then the index of $\rm{Rad}_\psi$ in $L$ equals the
index of the system of linear congruences
\begin{equation}\label{linear congruences reps}
\mcR(\bfl)\bfx\equiv0\mod(p^N)
\end{equation}
where $\bfx\in\Zp^d$, say. This index can be easily computed from the
elementary divisors of the matrix $\mcR(\bfl)$.  Recall that
$\mcR(\bfl)$ is said to have elementary divisor type
$\bfm=(m_1,\dots,m_d)\in[N]_0^d$ -- written $\nu(\mcR(\bfl))=\bfm$ --
if there are matrices $\beta,\gamma\in\text{GL}_d(\Z/p^N)$ such that
$$\beta\mcR(\bfl)\gamma\equiv\left(\begin{array}{ccc}p^{m_1}&&\\&\ddots&\\&&p^{m_d}\end{array}\right)$$
and $m_1\leq\dots\leq m_d$.  Given $N\in\N_0$ and $\bfm\in\N_0^d$ we
set
$$\mcN_{N,\bfm}:=\left|\left\{\bfl\in
  W_{p,N}|\,\nu(\mcR(\bfl))=\bfm\right\}\right|.$$ It is now easy to
  see that
\begin{equation}\label{formula reps poincare}
\zeta^{\rm{irr}}_{G,p}(s)=\sum_{N\in\N_0,\;\bfm\in\N_0^d}\mcN_{N,\bfm}p^{-\sum_{i\in[d]}(N-m_i)s/2}.
\end{equation}
This `Poincar\'e series' may, in analogy to equation~\eqref{equation
  igusa}, be expressed in terms of a $p$-adic integral. The integrand
of this (in general quite complicated) integral is defined in terms of
the minors of the matrix $\mcR(\bfy)$. This approach yields
immediately the rationality of (almost all of) the local
representation zeta functions of $\T$-groups, which was first
established in~\cite{HrushovskiMartin/07} by model-theoretic means
(and for all primes~$p$). The general case (of $\T$-groups of
arbitrary nilpotency class) is complicated by having to account for
over-counting when we run over the characters of $L'$. This can also
be formulated in terms of elementary divisors of matrices of
forms. See~\cite[Section 2.2]{Voll/06a} for details. We illustrate the
computations outlined above with a familiar example.

\begin{example} Let $G$ be the discrete Heisenberg group from Example \ref{example
heisenberg}. Here $d=2$ and $d'=1$. For all primes $p$ and $N\in\N_0$
  we have $W_{p,N}=(\Z/(p^N))^\times$. The commutator matrix
  $\mcR(\bfy)$ is given by
$$\mcR(y)=\left(\begin{array}{cc}&y\\-y&\end{array} \right)$$
and therefore 
$$\mcN_{N,\bfm}=\begin{cases}1&\text{ if }N=0,\\(1-p^{-1})p^N&\text{
if }N\in\N\text{ and }m_1=m_2=0,\\0&\text{ otherwise.}\end{cases}$$
Thus, for all primes $p$,
\begin{align*}
\zeta^{{\rm irr}}_{G,p}(s)&=\sum_{N\in\N_0, \bfm\in\N_0^2}\mcN_{N,\bfm}p^{-Ns+(m_1+m_2)s/2}\\
&=1+\sum_{N\in\N}(1-p^{-1})p^{(1-s)N}\\
&=(1-p^{-s})/(1-p^{1-s}),
\end{align*}
or, equivalently,
$$\zirr_G(s)=\sum_{m=1}^\infty\phi(m)m^{-s}=\zeta(s-1)\zeta(s)^{-1},$$
where $\phi$ denotes the Euler totient function. This was first proved
in \cite[Theorem 5]{NunleyMagid/89}, by entirely different means.
\end{example}

Notice that the local factors of the representation zeta function of
the Heisenberg group all satisfy the functional equation $$\zeta^{{\rm
    irr}}_{G,p}(s)|_{p\rightarrow p^{-1}}=p\,\zeta^{{\rm
    irr}}_{G,p}(s).$$ This generalises in the following way:
\begin{theorem}\cite[Theorem D]{Voll/06a}\label{theorem D} Let $G$ be a $\T$-group
  with derived group $G'$ of Hirsch length $d'$. Then, for almost
  all primes $p$,
\begin{equation}\label{funeq representations}
\nonumber\zirr_{G,p}(s)|_{p\rightarrow p^{-1}} =  p^{d'}\zirr_{G,p}(s).
\end{equation}
\end{theorem}

\subsubsection{Arithmetic groups}\label{subsubsection reps p-adic analytic groups}
Let $k$ be a number field with ring of integers $\mathcal{O}$ and let
$G=\bfG(\O_S)$ be an arithmetic lattice in a semisimple, simply
connected and connected $k$-defined algebraic group~$\bfG$ or, for
short, an arithmetic group. Recall that $G$ is said to have the
Congruence Subgroup Property (CSP) if every finite index subgroup of
$G$ is a congruence subgroup. (See~Section~3 of Nikolov's notes for
definitions of these terms.) Recall further that, if $G$ is rigid, the
representation zeta function of~$G$,
$$\zeta^{{\rm irr}}_G(s)=\sum_{n=1}^\infty r_n(G)n^{-s},$$ has finite
abscissa of convergence if and only if $G$ has polynomial
representation growth (PRG).

\begin{theorem}\cite[Theorems 1.2 and 1.3]{LubotzkyMartin/04}
Let $G$ be an arithmetic group. Then $G$ has PRG if and only it has the CSP.
\end{theorem}

Assume from now on that $G$ is an arithmetic group with the CSP.

\begin{proposition}\cite[Proposition~4.6]{LarsenLubotzky/08}
\label{proposition euler product reps}
There is a subgroup $G_0$ of $G$ of finite index in $G$ such that
\begin{equation}\label{euler product arithmetic groups}
\zirr_{G_0}(s)=\zirr_{\bfG(\C)}(s)^{|S_\infty|}\prod_{v\not\in
S}\zirr_{L_v}(s),
\end{equation}
 where $S_\infty$ denotes the set of archimedean valuations of $k$,
 $L_v$ is an open subgroup of $\bfG(\mathcal{O}_v)$ and
 $\zirr_{\bfG(\C)}(s)$ (resp. $\zirr_{L_v}(s)$) enumerates irreducible
 rational (resp. continuous) representations of $\bfG(\C)$
 (resp. $L_v$).
\end{proposition}

The fact that we need to pass to a finite index subgroup in
Proposition~\ref{proposition euler product reps} is insubstantial if we are mainly interested in the representation zeta function's abscissa of convergence. Indeed, we have the following

\begin{lemma}(\cite[Corollary 4.5]{LarsenLubotzky/08}) 
If $G_0$ is a finite index subgroup of the rigid PRG group
$G$, then the abscissae of convergence of the zeta functions
$\zirr_G(s)$ and $\zirr_{G_0}(s)$
coincide.
\end{lemma}

\begin{example} 
Let $G=\text{SL}_n(\Z)$. It is well-known that $\text{SL}_n(\Z)$
satisfies the CSP if and only if $n\geq 3$. In this case,
Proposition~\ref{proposition euler product reps} yields that
$$\zirr_{\text{SL}_n(\Z)}(s)=\zirr_{\text{SL}_n(\C)}(s)\prod_{p \text{
prime}}\zirr_{\text{SL}_n(\Zp)}(s).$$
\end{example}

Already at first glance the Euler product~\eqref{euler product
arithmetic groups} differs from the Euler factorisations we have
encountered before by the presence of a factor `at infinity'. The
Euler factor $\zirr_{\bfG(\C)}(s)$ is, however,
comparatively well understood. In particular, we know its abscissa of
convergence in certain cases.

\begin{theorem}\cite[Theorem 5.1]{LarsenLubotzky/08}\label{theorem abscissa infinity}
If $\bfG(\C)$ is defined as above then the abscissa of convergence of
$\zirr_{\bfG(\C)}(s)$ is equal to $\rho/\kappa$, where
$\rho=\rm{rk}(G)$ and $\kappa=|\Phi^+|$ is the number of positive
roots.
\end{theorem}

The proof of Theorem~\ref{theorem abscissa infinity} is based the fact
that the rational representations of these groups are combinatorially
parametrised by their highest weights;
see~\cite[Section~5]{LarsenLubotzky/08} for details.

\begin{example}\label{example sl2 reps}
The group $\text{SL}_2(\C)$ has a unique irreducible representation of each finite dimension. Thus
$$\zirr_{\text{SL}_2(\C)}=\sum_{m=1}^\infty
m^{-s}=\zeta(s).$$ Indeed, the abscissa of convergence of the Riemann
zeta function is $1=1/1=\rho/\kappa$.
\end{example}

\begin{theorem}\cite[Theorem 1.2]{Avni/08}\label{theorem avni}
Let $G$ be an arithmetic group which satisfies the CSP. Then the
abscissa of convergence of $\zirr_{G}(s)$ is a rational number.
\end{theorem}

The proof of this deep result uses sophisticated tools from algebraic
geometry, model theory and the representation theory of finite groups
of Lie type. We only remark that whilst its conclusion is analogous to
one of the conclusions of Theorem~\ref{theorem duSG analytic}, its
proof requires substantially different methods.

\subsubsection{Compact $p$-adic analytic groups}

The groups $L_v$ in Proposition~\ref{proposition euler product reps}
are compact $p$-adic analytic groups. Let, more generally, $G$ be a
finitely generated profinite group. It is well-known (cf.~Section 10.1
in Klopsch's lecture notes) that the number $r_n(G)$ of isomorphism
classes of continuous irreducible $n$-dimensional complex
representations of $G$ is finite if and only if $G$ is FAb, i.e.~if
and only if every open subgroup of $G$ has finite abelianisation.

\begin{theorem}\cite[Theorem 1.1]{Jaikin/06}\label{theorem jaikin JAMS}
Let $G$ be a compact FAb $p$-adic analytic group with $p>2$. Then
there are natural numbers $n_1,\dots,n_k$ and functions
$f_1(p^{-s}),\dots,f_k(p^{-s})$, rational in $p^{-s}$, such that
$$\zeta^{{\rm irr}}_G(s)=\sum_{i\in[k]}n_i^{-s}f_i(p^{-s}).$$
\end{theorem}

This deep result takes a more complicated form than the rationality
results for Euler factors we have met before. It should not surprise
us, however, that the representation zeta function of a $p$-adic
analytic group is not, in general, a rational function just in
$p^{-s}$: whereas the continuous representations of a pro-$p$ group
clearly all have dimension a power of $p$ (as they factor over finite
index normal subgroups of the group), a $p$-adic analytic group is
only \emph{virtually} pro-$p$, i.e.~it has a pro-$p$ subgroup of
finite index. The natural numbers $n_1,\dots,n_k$ in
Theorem~\ref{theorem jaikin JAMS} can be interpreted as the dimensions
of the representations of the quotient of $G$ by a normal, finite
index pro-$p$ subgroup.

As the work on representation zeta functions for $\T$-groups sketched
in Section~\ref{subsubsection reps T-groups}, the proof of
Theorem~\ref{theorem jaikin JAMS} is based on a Kirillov orbit method
for compact $p$-adic analytic groups.

Explicit examples of representation zeta functions of compact $p$-adic
groups are thin on the ground. In~\cite{Jaikin/06}, Jaikin gives the
example of $\zeta^{{\rm irr}}_{\text{SL}_2(\Zp)}(s)$ for odd
$p$. (Note, however, that the Euler product over the local factors
(including $p=2$ and `infinity') only counts `congruence
representations' of $\text{SL}_2$, as $\text{SL}_2$ does not satisfy
the CSP.) In~\cite{KlopschVoll/08}, formulae are developed for the
representation zeta functions of the principal congruence subgroups
$\text{SL}_3^k(\Zp)$ for all primes~$p$ and $k\in \N$ ($k\geq 2$ if
$p=2$), and the abscissa of convergence of
$\zirr_{\text{SL}_3(\Z)}(s)$ is determined. A result on functional
equations of representation zeta functions of pro-$p$ groups in
globally defined families can also be found in this paper (cf.~Theorem
10.3 in Klopsch's lecture notes).

\subsection{Further variations}\label{subsection other variations}

\subsubsection{Nilpotent groups} 
Besides the zeta functions counting all subgroups, normal subgroups
 and representations of a $\T$-group~$G$, people have studied the zeta
 functions enumerating subgroups of~$G$ which are isomorphic to~$G$
 (\cite{GSS/88}), the `pro-isomorphic' zeta functions enumerating
 subgroups whose profinite completion is isomorphic to the profinite
 completion of~$G$ (\cite{duSLubotzky/96, Berman/07}, and the zeta
 functions enumerating subgroups up to conjugacy
 (\cite[Section~3.2]{Voll/06a}). The last two satisfy Euler product
 decompositions into Euler factors which are rational in~$p^{-s}$.

\subsubsection{Compact $p$-adic analytic groups} 
Let $G$ be a compact $p$-adic analytic group. Recall that such a group
is virtually pro-$p$. In~\cite{duS/93} du Sautoy proved that the
`local' zeta function
$$\zeta_{G,p}(s)=\sum_{n=0}^\infty a_{p^n}(G)p^{-ns}$$ of $G$ is
rational in $p^{-s}$. He also proved that the `global' zeta function
$\zeta_{G}(s)$ counting all finite-index subgroups is rational in
$p^{-s}, n_1^{-s},\dots,n_{k}^{-s}$ for natural numbers
$n_1,\dots,n_k$ (analogous to Theorem~\ref{theorem jaikin JAMS}), and
established similar results for zeta functions counting normal
subgroups, $r$-generator subgroups and subgroups up to conjugacy in
compact $p$-adic analytic groups. We refer to \cite[Chapter
  16]{LubotzkySegal/03} for details. In~\cite{duS/05} du Sautoy showed
the rationality of certain generating functions enumerating the class
numbers of (i.e.~the total numbers of conjugacy classes in) families
of finite groups associated with compact $p$-adic analytic groups.

\subsubsection{Finite $p$-groups} 
The methods used to study the subgroup growth of nilpotent or $p$-adic
analytic groups have found applications in the enumeration of finite
$p$-groups. Given a prime $p$ and natural numbers $c$ and $d$, let
$f(n,p,c,d)$ denote the number of (isomorphism classes of)
$d$-generator $p$-groups of order $p^n$ and nilpotency class at most
$c$. We define the Dirichlet generating function
$$\zeta_{c,d,p}(s):=\sum_{n=0}f(n,p,c,d)p^{-ns}.$$ In~\cite{duS/02}
du Sautoy proved that these generating series are rational in the parameter
$p^{-s}$ (cf.~\cite[Section~16.4]{LubotzkySegal/03} for an
exposition). It follows easily from the structure theorem for finite
abelian $p$-groups that
\begin{equation}
\zeta_{1,d,p}(s)=\zeta_p(s)\zeta_p(2s)\cdots\zeta_p(ds).\label{formula
count abelian p}
\end{equation} In~\cite{Voll/02} it was proved that, for all primes
$p$,
$$\zeta_{2,2,p}(s)=\zeta_p(s)\zeta_p(2s)\zeta_p(3s)^2\zeta_p(4s).$$ No
other explicit formulae of this kind are known.

\section{Open problems and conjectures}\label{section open problems}

\subsection{Subring and subgroup zeta functions}

\begin{conjecture}\cite[p.~188]{GSS/88}
Let $F_{c,d}$ denote the free class-$c$-nilpotent group on $d$
  generators. Then $\zeta_{F_{c,d}}(s)$ and $\zeta^\nl_{F_{c,d}}(s)$
  are almost uniform, i.e.~there are rational functions $W_{c,d}(X,Y),
  W^\nl(X,Y)\in\Q(X,Y)$ such that, for almost all primes~$p$,
\begin{align*}
\zeta_{F_{c,d},p}(s)&=W_{c,d}(p,p^{-s})\\
\zeta^\nl_{F_{c,d},p}(s)&=W^\nl_{c,d}(p,p^{-s}).
\end{align*}
\end{conjecture}

\begin{conjecture}\label{degree conjecture} 
  Let $L$ be a class-$c$-nilpotent Lie ring of rank $n$ with upper
  central series $(Z_i(L))_i$, $i=0,\dots,c$. Set
  $n_i:=\text{rk}(L/Z_i(L))$ (so $n_0=n=\text{rk}(L)$). Then, for
  almost all primes $p$,
\begin{align}\label{degree conjecture I}
\text{deg}_{p^{-s}}(\zeta^\nl_{L,p}(s))&=-\sum_{i=0}^{c}n_i\\
\lim_{s\rightarrow
-\infty}(p^{-s})^{\sum_{i=1}^{c}n_i}\zeta^\nl_{L,p}(s)&=(-1)^np^{\binom{n}{2}}.\label{degree
conjecture II}
\end{align}
\end{conjecture}
Note that, for the primes $p$ for which $\zeta_{L,p}(s)$ satisfies a
functional equation of the form
\begin{equation}\label{equation funeq normal}
\zeta^\nl_{L,p}(s)|_{p\rightarrow
  p^{-1}}=(-1)^np^{\binom{n}{2}-s\sum_{i=0}^cn_i}\zeta^\nl_{L,p}(s),
\end{equation}
these are simple corollaries of~\eqref{equation funeq normal}. In
particular, Conjecture~\ref{degree conjecture} holds if $c\leq 2$
(cf.~\cite[Theorem C]{Voll/06a}). For higher classes, however, it is
known that the equation~\eqref{equation funeq normal} does not hold in
general. All known examples (cf., e.g., \cite{duSautoyWoodward/08})
nevertheless satisfy equations~\eqref{degree conjecture I} and
\eqref{degree conjecture II}.

\begin{problem} Characterise nilpotent Lie rings for which the functional
  equation~\eqref{equation funeq normal} holds for almost all
  primes~$p$.
\end{problem}

A `conjectural' characterisation has been given
in~\cite[Chapter~4]{duSautoyWoodward/08}. 

\begin{conjecture} 
Let $L$ be a class-$2$-nilpotent Lie ring with $\text{rk}(L/L')=d$, $\text{rk}(Z(L))=m$ and $\text{rk}(L/Z(L))=r$. Let $\alpha^\nl$ denote the abscissa of
  convergence of $\zeta^\nl_L(s)$. Then
  $$\alpha^\nl=\max_{k\in[m]}\left\{d,\frac{k(m+d-k)+1}{r+k}\right\}.$$
\end{conjecture}
That $\alpha^\nl$ is greater or equal to the right hand side was
proved in~\cite{Paajanen/07}. Equality has been proved, in particular,
for the free class-$2$-nilpotent groups $F_{2,d}$
in~\cite{Voll/05a}. More generally, we ask

\begin{problem}\label{problem abscissae} 
Given a ring $L$, determine the abscissae of convergence of its
subring and ideal zeta functions, respectively.
\end{problem}

\begin{problem}\label{problem poles} 
Given a ring $L$, determine (a small superset of) the natural numbers
  $a_i,b_i$ occurring in the denominators of its local (ideal) zeta
  functions~(cf.~Theorem~\ref{theorem duS denominators}).
\end{problem}

It follows from~\cite{duSG/00} that the abscissa of convergence of a
ring's global zeta function is a simple function of these
integers. Problem~\ref{problem poles} is thus strictly harder than
Problem~\ref{problem abscissae}. Even (partial) answers for specific
families of Lie rings as nilpotent or soluble Lie rings or `simple'
Lie rings like $\mathfrak{sl}_n(\Z)$ would be very interesting.

\subsection{Representation zeta functions}

\begin{problem}(\cite[Problem 4.2]{LarsenLubotzky/08}) 
Characterise rigid groups, and groups of polynomial representation
growth (PRG).
\end{problem}

\begin{problem}
Let $G$ be a $\T$-group with representation zeta function
$\zirr_G(s)$. Is the abscissa of convergence of $\zirr_G(s)$ a
rational number? Does $\zirr_G(s)$ admit analytic continuation beyond
its abscissa of convergence?  Interpret the abscissa of convergence
and the poles of the Euler factors of $\zirr_G(s)$ in terms of the
structure of $G$.
\end{problem}


A positive answer to this problem would imply asymptotic
statements about the numbers of twist-isoclasses of representations of
$\T$-groups, analogous to Part~B of Theorem~\ref{theorem duSG
analytic}. 


\begin{problem}
  Let $\SLn^k(\Zp):=\ker(\SLn(\Zp)\rightarrow\SLn(\Z/(p^k\Z))$ denote
  the $k$-th congruence subgroup of $\text{SL}_n(\Zp)$. How do the
  functions $\zirr_{\SLn^k(\Zp)}(s)$ vary with the prime~$p$? What are
  the abscissae of convergence of the zeta functions
  $\zirr_{\SLn^k(\Z)}(s)$? What about other `classical' $p$-adic
  analytic groups?
\end{problem}

\begin{problem}
Let $G$ be an arithmetic group satisfying the CSP. Does its representation
zeta function $\zirr_G(s)$ admit analytic continuation beyond its
rational (Theorem~\ref{theorem avni}) abscissa of convergence?
\end{problem}

Again, a positive answer would give us control over the asymptotic of
the numbers $r_n(G)$ as $n$ tends to infinity.

\section{Exercises}

\begin{exercise}
Let $q$ be a prime power, $n\in\N$, and $I\subseteq[n-1]$. Show that
the number of flags of type $I$ in $\Fp^n$ is equal to
$\binom{n}{I}_q$.
\end{exercise}

\begin{exercise} 
Prove equation~\eqref{generating function heisenberg} directly.
\end{exercise}

\begin{exercise}(cf.~Example~\ref{example heisenberg normal}) 
Let $p$ be a prime, and $L_p=\Zp\otimes L$, where $L$ is the
Heisenberg Lie ring. In the setup of Section~\ref{subsection local
funeqs}, show that a coset $\Gamma M$ corresponds to an ideal if and
only if $M_{33}\mid M_{11}$ and $M_{33}\mid M_{22}$. Deduce that, for
all primes,
$$\zeta^\triangleleft_{L_p}(s)=\sum_{H\triangleleft_fL_p}|L_p:H|^{-s}=\frac{1}{(1-p^{-s})(1-p^{1-s})(1-p^{2-3s})}.$$
\end{exercise}

\begin{exercise} [$\star$]
Let $L$ be a ring of additive rank~$n$. Using the setup and notation
of Section~\ref{subsection counting}, show that a matrix
$M=(M_{ij})\in\text{Tr}_3(\Zp)$ encodes the generators of an
\emph{ideal} if and only if
\begin{equation*}
\forall i\in[n]:\; D \alpha^{-1}\mcR(\alpha[i])\equiv 0 \mod D_{ii},
\end{equation*}
(This is the `ideal'-analogue of equation~\eqref{subalgebra conditionII}.)
\end{exercise}

\begin{exercise} [$\star$] For $n\in\N$, let $L(n)=\Z^n$, considered as a ring with component-wise multiplication. Show that, for all primes $p$,
\begin{equation*}
\zeta_{L(2),p}(s)=\frac{(1+p^{-s})^2}{(1-p^{-s})(1-p^{1-3s})}.
\end{equation*}
Show that, for all $n\in\N$ and all primes $p$,
$$\zeta_{L(n),p}(s)=\zeta_p(s)^n,$$
where $\zeta_p(s)=(1-p^{-s})^{-1}$.
\end{exercise}

\begin{exercise} [$\star$]
Let $G$ be the group defined in Example~\ref{example nonuniform}. Show that, for $p\neq2$, 
$$\zirr_{G,p}(s)=W_1(p,p^{-s}) + b(p) W_2(p,p^{-s}),$$
where 
\begin{equation*}
W_1(X_1,X_2)=\frac{1-X_2^{3}}{1-X_1^3X_2^3},\quad
W_2(X_1,X_2)=\frac{(X_1-1)(X_2-1)X_2^2}{(1-X_1^2X_2^2)(1-X_1^3X_2^3)}
\end{equation*}
and $b(p)$ is defined as in Example~\ref{example ec
projective}. Deduce the assertion of Theorem~\ref{theorem D} in these
cases.
\end{exercise}

\begin{exercise} Establish formula~\eqref{formula count abelian p}.
\end{exercise}

\begin{acknowledgements} 
I am indebted to Mark Berman, Benjamin Klopsch and Alexander
Stasinski, whose careful comments greatly improved these notes.
\end{acknowledgements}

\bibliographystyle{amsplain}
\providecommand{\bysame}{\leavevmode\hbox to3em{\hrulefill}\thinspace}
\providecommand{\MR}{\relax\ifhmode\unskip\space\fi MR }
\providecommand{\MRhref}[2]{%
  \href{http://www.ams.org/mathscinet-getitem?mr=#1}{#2}
}
\providecommand{\href}[2]{#2}


\end{document}